\documentclass[11pt]{amsart}

\usepackage{epsfig, color}
\usepackage{esint}

\usepackage{amsmath}
\usepackage{amsthm}
\usepackage{MnSymbol}
\usepackage{hyperref}
\usepackage[margin=1.0in]{geometry}

\numberwithin{equation}{section}

\newtheorem{prop}{Proposition}
\newtheorem{lemma}[prop]{Lemma}

\newtheorem{thm}[prop]{Theorem}
\newtheorem{cor}[prop]{Corollary}

\numberwithin{prop}{section}

\theoremstyle{definition}
\newtheorem{defn}[prop]{Definition}

\newtheorem{rmk}[prop]{Remark}

\newcommand{\del}{\partial}
\newcommand{\dt}{\frac{\partial}{\partial t}}
\newcommand{\brs}[1]{\left| #1 \right|}

\newcommand{\gG}{\Gamma}
\renewcommand{\gg}{\gamma}
\newcommand{\gD}{\Delta}
\newcommand{\gd}{\delta}
\newcommand{\gs}{\sigma}

\newcommand{\gl}{\lambda}
\newcommand{\gL}{\Lambda}

\newcommand{\ga}{\alpha}
\newcommand{\gb}{\beta}
\renewcommand{\ge}{\epsilon}
\newcommand{\N}{\nabla}

\newcommand{\ea}{\eta^{\alpha}}

\renewcommand{\bar}[1]{\overline{#1}}

\newcommand{\IP}[1]{\left<#1\right>}

\DeclareMathOperator{\Rc}{Rc}

\newcommand{\gU}{\Upsilon}

\newcommand{\Gpos}{\Gamma_1^{+}}
\newcommand{\Gneg}{\Gamma_1^{-}}
\newcommand{\ddt}{\frac{\partial}{\partial t}}

\newcommand{\dbrs}[1]{\brs{\brs{#1}}}
\newcommand{\dIP}[1]{\IP{\IP{#1}}}

\begin{document}

\title[Riemannian structure on conformal classes and inverse Gauss curvature flow]{A formal Riemannian structure on conformal classes and the inverse Gauss curvature flow}

\author{Matthew Gursky}
\address{Department of Mathematics
         University of Notre Dame\\
         Notre Dame, IN 46556}
\email{\href{mailto:mgursky@nd.edu}{mgursky@nd.edu}}

\author{Jeffrey Streets}
\address{Department of Mathematics\\
         University of California\\
         Irvine, CA  92617}
\email{\href{mailto:jstreets@math.uci.edu}{jstreets@math.uci.edu}}

\date{June 16th, 2015}

\begin{abstract} We define a formal Riemannian metric on a given conformal class of metrics on a closed Riemann surface.  We show interesting formal properties for this metric, in particular the curvature is nonpositive and the Liouville energy is geodesically convex.  The geodesic equation for this metric corresponds to a degenerate elliptic fully nonlinear PDE, and we prove that any two points are connected by a $C^{1,1}$ geodesic.  Using this we can define a length space structure on the given conformal class.  We present a different approach to the uniformization theorem by studying the negative gradient flow of the normalized Liouville energy, a new geometric flow whose principal term is the inverse of the Gauss curvature.  We prove long time existence of solutions with arbitrary initial data and weak convergence to constant scalar curvature metrics.  This is all a special case of a more general construction on even dimensional manifolds related to the $\sigma_{\frac{n}{2}}$-Yamabe problem, which will appear in \cite{GS}.
\end{abstract}

\maketitle

\section{Introduction}

In this paper we define a formal Riemannian metric on the set of metrics in a conformal
class with positive (or negative) curvature.  Namely, let $(M, g_0)$ be a compact Riemannian
surface with positive Gauss curvature $K_0 > 0$, and let $[g_0]$ denote the conformal class of $g_0$.  Define
\begin{align} \label{conedefintro}
\Gpos = \{ g_u = e^{2u}g_0 \in [g_0] \ :\ K_u = K_{g_{u}} > 0 \},
\end{align}
the space conformal metrics with positive Gauss curvature.
Formally, the tangent space to $[g_0]$ at any metric $g_u \in [g_0]$ is given by $C^{\infty}(M)$.
Let $K_u$ denote the Gauss curvature of $g_u \in \Gpos$.  We define for $\phi,\psi \in C^{\infty}(M)$ (cf. Definition \ref{metdef}),
\begin{align} \label{gskmetric}
 \llangle \phi, \psi \rrangle_u =&\ \int_M \phi \psi K_u dA_u,
\end{align}
where $dA_u$ is the area form of $g_u$.
In other words, we weight the standard $L^2$ metric with the Gauss curvature
of the given conformal metric.  If the Gauss curvature of $g_0$ is negative, we define
\begin{align} \label{negconeintro}
\Gneg = \{ g_u = e^{2u}g_0 \in [g_0] \ :\ K_u = K_{g_{u}} < 0 \},
\end{align}
and the metric associated to this space is given by
\begin{align} \label{gskmetricneg}
 \llangle \phi, \psi \rrangle_u =&\ \int_M \phi \psi (-K_u) dA_u.
\end{align}

This definition is loosely inspired by the
Mabuchi-Semmes-Donaldson \cite{MabuchiSymp, Semmes, Donaldson} metric of
K\"ahler geometry, wherein a formal
Riemann metric is put on a K\"ahler class by imposing on the tangent space to a
given K\"ahler potential the $L^2$ metric with respect to the associated
K\"ahler metric.  As observed in \cite{MabuchiSymp}, this metric enjoys many
nice formal
properties, for instance nonpositive sectional curvature.  Moreover, it has a
profound relationship to natural functionals in K\"ahler geometry such as the
Mabuchi $K$-energy and the Calabi energy, as well as their gradient flow, the
Calabi flow.  Based on these excellent formal properties Donaldson proposed a
series of conjectures on the existence of geodesics, geodesic rays, as well as
the existence properties of the Calabi flow.  The tremendous work of many
authors (an incomplete list of references is \cite{Blocki, CalabiChen, ChenSOKM, ChenSOKM3, ChenTian,Darvas, Guedj, PS,RW}) has resulted in
the verification of many of these conjectures,
which largely centers around a detailed analysis of the very delicate geodesic
equation, which can be interpreted as a degenerate Monge-Ampere equation.

As we will see, there is a tight analogy in many respects between the Mabuchi
metric and the metric defined in (\ref{gskmetric}).   In section \ref{MCC} we establish
a formal path derivative which can be regarded as the Levi-Civita connection
associated to the metric.  Using these we compute the sectional curvature, and
show that the metric is nonpositively curved.  Next, in section \ref{GLE} we derive the geodesic equation.  Formal calculations derived using either
the path derivative or variations of the length functional yield that a
one-parameter family of conformal factors $u : [a,b] \rightarrow \Gpos$ is a geodesic if and only if
\begin{align} \label{geodesics}
\frac{\partial^2 u}{\partial t^2} + \frac{ \displaystyle |\nabla_0 \frac{\partial u}{\partial t}|^2}{\displaystyle K_0 - \Delta_0 u} = 0.
\end{align}
We end section \ref{GLE} with the fundamental observation that one parameter families of conformal
transformations are automatically geodesics (Proposition
\ref{conformalpathsgeodesics}).

Section \ref{geoexist} contains the proof of the existence of $C^{1,1}$ geodesics connecting any two points in $\gG_1^+$.  Equation (\ref{geodesics}) turns out to be a fully nonlinear, degenerate elliptic equation.  We study a natural regularization of (\ref{geodesics}) which renders it a convex, strictly elliptic equation.  By maximum principle arguments we establish a priori $C^{1,1}$ estimates independent of the regularization parameter, which yield the existence of the $C^{1,1}$ solution as claimed.  With this in place in section \ref{metricspace} we rigorously show that the length of the unique regularizable geodesic connecting any two points does indeed define a metric space structure $(\gG_1^+,d)$ (Corollary \ref{metricspacecor}), and that this metric space is nonpositively curved in the sense of Alexandrov (Proposition \ref{NPCprop}).  The situation is summarized in the following theorem.

\begin{thm} \label{msthm} Let $(M^2, g_0)$ be a compact Riemann surface.  Then $(\gG_1^{\pm}, d)$ is a length space, with any two points connected by a unique regularizable $C^{1,1}$ geodesic.  Moreover, it is nonpositively curved in the sense of Alexandrov.
\end{thm}

Furthering the analogy with the K\"ahler setting, the metric (\ref{gskmetric}) is closely
associated with the gradient flow of the normalized Liouville energy.  Previously Osgood-Phillips-Sarnack \cite{OPS} studied the negative gradient flow, but with respect to the $L^2$ metric, yielding an equation which is similar to Ricci flow.  With the ambient geometry given by the weighted $L^2$ metric on $\gG_1^+$, we arrive at a different evolution equation, expressed in terms of the conformal factor as
\begin{align*}
\frac{\del u}{\del t} =&\ - 1 + \frac{\bar{K_u}}{K_u},
\end{align*}
where $\bar{K}$ is the average Gauss curvature.  This is a fully nonlinear parabolic equation for $u$.  On $\gG_1^-$ we arrive at
\begin{align*}
\frac{\del u}{\del t} =&\ 1 - \frac{\bar{K_u}}{K_u}.
\end{align*}
Generically we will refer to these as \emph{inverse Gauss curvature flow}.  Our primary results are as follows:

\begin{thm} \label{flowthm} Fix $(M^2, g)$ a compact Riemann surface and $u \in \gG_1^{\pm}$.
\begin{enumerate}
\item The solution to IGCF with initial condition $u$ exists on $[0,\infty)$.
\item The normalized Liouville energy is convex in time along the flow line, i.e.
\begin{align*}
\frac{d^2}{dt^2} F[u(t)] \geq 0.
\end{align*}
\item Given $v(x,t)$ another solution to IGCF, the distance between flow lines is nonincreasing, i.e.
\begin{align*}
\frac{d}{dt} d(u(t), v(t)) \leq 0.
\end{align*}
\item If $u \in \gG_1^-$, then the solution converges as $t \to \infty$ in the $C^{\infty}$ topology to the unique conformal metric of constant scalar curvature.
\item If $u \in \gG_1^+$ and $(M^2, g) \cong (S^2, g_{S^2})$, then the solution converges weakly in the distance topology to a minimizer for $F$ in the completion $(\bar{\gG}_1^+, \bar{d})$.
\end{enumerate}
\end{thm}

\begin{rmk} Properties (2) and (3) are directly analogous to results relating the $K$-energy, Mabuchi metric, and Calabi flow (cf. \cite{CalabiChen}).  We emphasize that the point of the hypothesis $(M^2, g) \cong (S^2, g_{S^2})$ is that we are NOT yet able to use the IGCF to provide an a priori proof of the Uniformization Theorem.  We require the existence of a constant scalar curvature metric to ensure the convergence of the flow in the distance topology.
\end{rmk}

\begin{rmk}  Although our results are in the setting of two dimensions, this is actually a special case of a more general construction on even dimensional manifolds.  In dimensions $n \geq 4$, one can define a Riemannian structure on subsets of conformal classes satisfying an admissibility condition which naturally arises in the study of the $\sigma_{\frac{n}{2}}$-Yamabe problem.  As in the case of surfaces, the underlying metric is closely associated to a functional whose critical points `uniformize' the conformal class.  This will be presented in a forthcoming article \cite{GS}.   
\end{rmk}

\section{Metric, connection, and curvature}  \label{MCC}

In this section we define the formal Riemannian metric on the space of conformal metrics of positive/negative curvature.  Because of the dependence on the
sign of the curvature, we will first consider the positive case in detail, then provide the corresponding results for metrics of negative curvature without
(or at most cursory) proofs.

\subsection{The positive cone}  Let $(M,g_0)$ be a closed surface with positive Gauss curvature, and let
\begin{align} \label{conedef}
\Gpos = \{ g_u = e^{2u}g_0 \ :\ K_u = K_{g_{u}} > 0 \}
\end{align}
denote the space conformal metrics with positive Gauss curvature.  The formal tangent space at $g_u$ is
\begin{align} \label{TG}
T_u \Gpos \cong C^{\infty}(M).
\end{align}

\begin{defn} For $\alpha, \beta \in T_u \Gpos \cong C^{\infty}(M)$, define
\begin{align} \label{metdef}
\llangle \alpha, \beta \rrangle_u = \int_M \alpha \beta K_u dA_u,
\end{align}
where $dA_u$ is the volume form of the metric $g_u = e^{2u}g_0$.
\end{defn}

\begin{rmk} To simplify notation, we will often write $u \in \Gpos$ to mean $g_u = e^{2u}g_0 \in \Gpos$.  \end{rmk}

Given a path of conformal factors $u : [a,b] \rightarrow \Gpos$ and a vector field $\alpha = \alpha(\cdot,t)$ along $u$, we define
\begin{align} \label{Dtdef}
\frac{D}{\partial t}\alpha = \frac{\partial}{\partial t}\alpha + \frac{1}{K_u}\langle \nabla_u \alpha, \nabla_u \frac{\partial u}{\partial t} \rangle_u,
\end{align}
where $\langle \cdot,\cdot \rangle_u$ denotes the inner product with respect to $g_u$.

\begin{lemma} \label{comppos} The connection defined by (\ref{Dtdef}) is metric-compatible and torsion-free.  \end{lemma}

\begin{proof}  Let $\alpha, \beta$ be vector fields along the path $u : [a,b] \rightarrow \Gpos$. To simplify notation we will drop the subscript $u$, and all metric-dependent
quantities (curvature, area form, etc.) will be understood to be with respect to $g_u$.

We first prove compatibility.  This will require us to record the standard variational formulas for a path of conformal metrics $g = g(t) = e^{2u}g_0$:
\begin{align} \label{dots} \begin{split}
\frac{\partial}{\partial t}K &= - \Delta \big( \frac{\partial u}{\partial t} \big) - 2 K \frac{\partial u}{\partial t}, \\
 \frac{\partial}{\partial t} dA &= 2 \frac{\partial u}{\partial t} dA.
 \end{split}
 \end{align}
Then
\begin{align*}
\frac{d}{dt} \llangle \ga, \gb \rrangle =&\ \frac{d}{dt} \int_M \ga \gb K dA \\
=&\ \llangle \frac{\partial}{\partial t}{\ga}, \gb \rrangle_u  + \llangle \ga,\ddt \gb \rrangle_u + \int_M \ga \gb \ddt \big( K dA \big)\\
=&\ \llangle \frac{\partial}{\partial t}{\ga}, \gb \rrangle_u + \llangle \ga,\ddt \gb \rrangle_u - \int_M \ga \gb \Delta \big( \frac{\partial u}{\partial t} \big) \\
=&\ \llangle \frac{\partial}{\partial t}{\ga}, \gb \rrangle_{u} + \llangle \ga,\ddt \gb \rrangle_{u} + \int_M \langle \nabla \ga \gb + \alpha \nabla \beta , \nabla  \frac{\partial u}{\partial t} \rangle dA \\
=&\ \llangle \frac{D}{\del t} \ga,\gb \rrangle_{u} + \llangle \ga,\frac{D}{\del t} \gb \rrangle_{u}.
\end{align*}
To compute the torsion, let $u = u(\cdot, s, t)$ be a two-parameter family of conformal factors in $\Gpos$. Then
\begin{align*}
\frac{D}{\del s} \frac{\del u}{\del t} - \frac{D}{\del t} \frac{\del u}{\del
s} =&\ \frac{\del^2 u}{\del s \del t} +   \frac{1}{K_u}\langle \nabla \frac{\partial u}{\partial t} , \nabla \frac{\partial u}{\partial s} \rangle_u - \frac{\del^2 u}{\del s \del t}  - \frac{1}{K_u}\langle \nabla \frac{\partial u}{\partial s} , \nabla \frac{\partial u}{\partial t} \rangle_u\\
=&\ 0.
\end{align*}

\end{proof}

\begin{prop}   \label{SecPosProp} Given $\phi, \psi \in T_u \Gamma_1^{+}$, the sectional curvature of the plane in $T_u \Gamma_1^{+}$ spanned by $\phi, \psi$ is given by
\begin{align} \label{K2form} \begin{split}
K(\phi,\psi) &= \int \frac{1}{K_u} \Big\{ - |\nabla \phi|_u^2 |\nabla \psi|_u^2 + g_u(\nabla \phi, \nabla \psi)^2 \Big\} dA_u \\
&= - \int \frac{1}{K_u} \big| d \phi \wedge d \psi |_u^2\ dA_u \\
&\leq 0.
\end{split}
\end{align}
\end{prop}

\begin{proof} Let $u = u(s,t)$ be a 2-parameter family of conformal factors, and $\alpha = \alpha(s,t) \in T_{u(s,t)}\Gamma_1^{+}$.
Using the formula for the connection, we have
\begin{align} \begin{split} \label{2secD1}
\frac{D}{\partial s} \frac{D}{\partial t} \alpha &= \frac{\partial}{\partial s} \big( \frac{D}{\partial t} \alpha \big) +  \frac{1}{K_u} g_u \big( \nabla ( \frac{D}{\partial t} \alpha) , \nabla ( \frac{\partial u }{\partial s}) \big) \\
&= \frac{\partial}{\partial s} \Big\{ \frac{\partial \alpha}{\partial t} + \frac{1}{K_u} g_u\big(  \nabla \alpha , \nabla( \frac{\partial u}{\partial t}) \big) \Big\}  + \frac{1}{K_u} g_u \big( \nabla \big\{ \frac{\partial \alpha}{\partial t} + \frac{1}{K_u} g_u\big(  \nabla \alpha , \nabla( \frac{\partial u}{\partial t}) \big) \big\}, \nabla ( \frac{\partial u }{\partial s}) \big)  \\
&= I + II.
\end{split}
\end{align}
In the following, we will omit the subscript $u$, and all metric-dependent quantities will be understood to be with respect to $g_u$.

To evaluate $I$, we will need the variational formulas (\ref{dots}) along with
\begin{align} \label{Kgs}
\frac{\partial}{\partial s}g(\nabla f_1, \nabla f_2) = - 2 \frac{\partial u}{\partial s} g(\nabla f_1, \nabla f_2), \quad f_1, f_2 \in C^{\infty}(M).
\end{align}

It follows that
\begin{align} \label{SS1} \begin{split}
I &= \frac{\partial}{\partial s} \Big\{ \frac{\partial \alpha}{\partial t} + \frac{1}{K} g\big(  \nabla \alpha , \nabla( \frac{\partial u}{\partial t}) \big) \Big\} \\
&= \frac{\partial^2 \alpha}{\partial s \partial t} - \frac{1}{K^2} \frac{\partial K}{\partial s} g \big( \nabla \alpha, \nabla (\frac{\partial u}{\partial t}) \big) + \frac{1}{K} \frac{\partial}{\partial s}g \big( \nabla \alpha, \nabla (\frac{\partial u}{\partial t}) \big) + \frac{1}{K} g \big( \nabla( \frac{\partial \alpha}{\partial s}), \frac{\partial u}{\partial t}) \big) \\
& \quad + \frac{1}{K} g \big( \nabla \alpha, \nabla ( \frac{\partial^2 u}{\partial s \partial t}) \big) \\
&= \frac{\partial^2 \alpha}{\partial s \partial t} + \frac{1}{K^2}\big[ \Delta\big( \frac{\partial u}{\partial s}\big) + 2 K \frac{\partial u}{\partial s} \big] g\big( \nabla \alpha, \nabla (\frac{\partial u}{\partial t}) \big)
- 2 \frac{1}{K} \frac{\partial u}{\partial s} g\big( \nabla \alpha, \nabla (\frac{\partial u}{\partial t}) \big) \\
& \quad + \frac{1}{K} g \big( \nabla( \frac{\partial \alpha}{\partial s}), \frac{\partial u}{\partial t}) \big)  + \frac{1}{K} g \big( \nabla \alpha, \nabla ( \frac{\partial^2 u}{\partial s \partial t}) \big)  \\
&= \frac{\partial^2 \alpha}{\partial s \partial t} + \frac{1}{K^2} \Delta\big( \frac{\partial u}{\partial s} \big)  g\big( \nabla \alpha, \nabla (\frac{\partial u}{\partial t}) \big)
+ \frac{1}{K} g \big( \nabla( \frac{\partial \alpha}{\partial s}), \frac{\partial u}{\partial t}) \big)  + \frac{1}{K} g \big( \nabla \alpha, \nabla ( \frac{\partial^2 u}{\partial s \partial t}) \big).
\end{split}
\end{align}
\ \

Turning to $II$, we write
\begin{align} \label{SS2} \begin{split}
II &= \frac{1}{K} g \big( \nabla \big\{ \frac{\partial \alpha}{\partial t} + \frac{1}{K} g_u\big(  \nabla \alpha , \nabla( \frac{\partial u}{\partial t}) \big) \big\}, \nabla ( \frac{\partial u }{\partial s}) \big)  \\
&= \frac{1}{K} g \big( \nabla (\frac{\partial \alpha}{\partial t}), \nabla ( \frac{\partial u}{\partial s}) \big) + \frac{1}{K} g \big( \nabla \big\{ \frac{1}{K} g \big( \nabla \alpha, \nabla ( \frac{\partial u}{\partial t} ) \big) \big\} , \nabla ( \frac{\partial u}{\partial s} ) \big).
\end{split}
\end{align}
\ \

Combining (\ref{SS1}) and (\ref{SS2}) and skew-symmetrizing in $s,t$, we have
\begin{align} \label{SS3} \begin{split}
\Big( \frac{D}{\partial s} \frac{D}{\partial t} - \frac{D}{\partial t} \frac{D}{\partial s} \Big) \alpha &= \frac{1}{K} \Big\{ \frac{1}{K} \Delta\big( \frac{\partial u}{\partial s} \big)  g\big( \nabla \alpha, \nabla (\frac{\partial u}{\partial t}) \big) - \frac{1}{K} \Delta\big( \frac{\partial u}{\partial t} \big)  g\big( \nabla \alpha, \nabla (\frac{\partial u}{\partial s}) \big)  \\
& \hskip-.5in + g \big( \nabla \big\{ \frac{1}{K} g \big( \nabla \alpha, \nabla ( \frac{\partial u}{\partial t} ) \big) \big\} , \nabla ( \frac{\partial u}{\partial s} ) \big)
- g \big( \nabla \big\{ \frac{1}{K} g \big( \nabla \alpha, \nabla ( \frac{\partial u}{\partial s} ) \big) \big\} , \nabla ( \frac{\partial u}{\partial t} ) \big) \Big\}.
\end{split}
\end{align}
To compute the sectional curvature of the plane spanned by $\{ \frac{\partial u}{\partial s}, \frac{\partial u}{\partial t} \}$, we take $\alpha = \frac{\partial u}{\partial t}$ in the formula above, then take the inner product with $\frac{\partial u}{\partial s}$:
\begin{align} \begin{split} \label{SS4}
 \Big\langle \Big( \frac{D}{\partial s} \frac{D}{\partial t} - &\frac{D}{\partial t} \frac{D}{\partial s} \Big) \frac{\partial u}{\partial t}, \frac{\partial u}{\partial s} \Big\rangle_u  = \\ &
\hskip-.6in \int \Big\{ \frac{1}{K} \frac{\partial u}{\partial s} \Delta\big( \frac{\partial u}{\partial s} \big)  g\big( \nabla ( \frac{\partial u}{\partial t} ), \nabla (\frac{\partial u}{\partial t}) \big) - \frac{1}{K}\frac{\partial u}{\partial s} \Delta\big( \frac{\partial u}{\partial t} \big)  g\big( \nabla ( \frac{\partial u}{\partial t} ), \nabla (\frac{\partial u}{\partial s}) \big)  \\
& \hskip-.5in + \frac{\partial u}{\partial s} g \big( \nabla \big\{ \frac{1}{K} g \big( \nabla ( \frac{\partial u}{\partial t} ), \nabla ( \frac{\partial u}{\partial t} ) \big) \big\} , \nabla ( \frac{\partial u}{\partial s} ) \big)
- \frac{\partial u}{\partial s} g \big( \nabla \big\{ \frac{1}{K} g \big( \nabla ( \frac{\partial u}{\partial t} ), \nabla ( \frac{\partial u}{\partial s} ) \big) \big\} , \nabla ( \frac{\partial u}{\partial t} ) \big) \Big\}
\end{split}
\end{align}
If we integrate by parts in the last two terms, we find
\begin{align} \label{SS5} \begin{split}
\int  & \Big\{ \frac{\partial u}{\partial s} g \big( \nabla \big\{ \frac{1}{K} g \big( \nabla ( \frac{\partial u}{\partial t} ), \nabla ( \frac{\partial u}{\partial t} ) \big) \big\} , \nabla ( \frac{\partial u}{\partial s} ) \big)
- \frac{\partial u}{\partial s} g \big( \nabla \big\{ \frac{1}{K} g \big( \nabla ( \frac{\partial u}{\partial t} ), \nabla ( \frac{\partial u}{\partial s} ) \big) \big\} , \nabla ( \frac{\partial u}{\partial t} ) \big) \Big\}  \\
&= \int \Big\{ - \frac{1}{K}\frac{\partial u}{\partial s} \Delta\big( \frac{\partial u}{\partial s} \big)  g\big( \nabla ( \frac{\partial u}{\partial t} ), \nabla (\frac{\partial u}{\partial t}) \big) - \frac{1}{K} g \big( \nabla ( \frac{\partial u}{\partial s} ), \nabla ( \frac{\partial u}{\partial s} ) \big) g \big( \nabla ( \frac{\partial u}{\partial t} ), \nabla ( \frac{\partial u}{\partial t} ) \big) \\
& \quad + \frac{1}{K}\frac{\partial u}{\partial s} \Delta\big( \frac{\partial u}{\partial t} \big)  g\big( \nabla ( \frac{\partial u}{\partial s} ), \nabla (\frac{\partial u}{\partial t}) \big) + \frac{1}{K} g \big( \nabla ( \frac{\partial u}{\partial t} ), \nabla ( \frac{\partial u}{\partial s} ) \big) g \big( \nabla ( \frac{\partial u}{\partial t} ), \nabla ( \frac{\partial u}{\partial s} ) \big) \Big\}.
\end{split}
\end{align}
Substituting this into (\ref{SS4}), we have
\begin{align} \begin{split} \label{SS6}
 \Big\langle \Big( \frac{D}{\partial s} \frac{D}{\partial t} & - \frac{D}{\partial t} \frac{D}{\partial s} \Big) \frac{\partial u}{\partial t}, \frac{\partial u}{\partial s} \Big\rangle_u  \\
  & = \int \frac{1}{K} \Big\{ - g \big( \nabla ( \frac{\partial u}{\partial s} ), \nabla ( \frac{\partial u}{\partial s} ) \big) g \big( \nabla ( \frac{\partial u}{\partial t} ), \nabla ( \frac{\partial u}{\partial t} ) \big) + g \big( \nabla ( \frac{\partial u}{\partial t} ), \nabla ( \frac{\partial u}{\partial s} ) \big)^2 \Big\}  \\
&\leq 0,
\end{split}
\end{align}
as claimed.

\end{proof}

\subsection{The negative cone}  Now assume $(M,g_0)$ is a closed surface with $K_0 < 0$, and
let
\begin{align} \label{conenegdef}
\Gneg = \{ g_w = e^{2w}g_0 \ :\ K_w = K_{g_{w}} < 0 \}
\end{align}
denote the space conformal metrics with negative Gauss curvature.  The formal tangent space at $g_w$ is
\begin{align} \label{TG1}
T_u \Gneg \cong C^{\infty}(M).
\end{align}

\begin{defn} For $\alpha, \beta \in T_u \Gneg \cong C^{\infty}(M)$, define
\begin{align} \label{metdefneg}
\llangle \alpha, \beta \rrangle_{u} = \int_M \alpha \beta (-K_w) dA_w,
\end{align}
where $dA_w$ is the area form of the metric $g_w = e^{2w}g_0$.
\end{defn}

As before, we write $w \in \Gneg$ to mean $g_w = e^{2w}g_0 \in \Gneg$.  Given a path of conformal factors $w : [a,b] \rightarrow \Gpos$ and a vector field $\alpha = \alpha(\cdot,t)$ along $u$, we now define
\begin{align} \label{Dtdefneg}
\frac{D}{\partial t}\alpha = \frac{\partial}{\partial t}\alpha + \frac{1}{K_w}\langle \nabla_w \alpha, \nabla_w \frac{\partial w}{\partial t} \rangle_w.
\end{align}
The proof of the next two results are essentially the same as in the case of the positive cone:

\begin{lemma} \label{compneg} The connection defined by (\ref{Dtdefneg}) is metric-compatible and torsion-free.  \end{lemma}

\begin{prop}   \label{SecNegProp} Given $\phi, \psi \in T_w \Gneg$, we have
\begin{align} \label{K2formneg} \begin{split}
K(\phi,\psi) &= \int \frac{1}{(-K_w)} \Big\{ - |\nabla \phi|_w^2 |\nabla \psi|_w^2 +  g_w(\nabla \phi, \nabla \psi)^2 \Big\} dA_w \leq 0.
\end{split}
\end{align}
\end{prop}

\section{Geodesics, length, and energy}  \label{GLE}

Using the definition of the positive cone metric in (\ref{metdef}) we can also define the associated notions of energy and length.

\begin{defn} Given a path $u: [a,b] \rightarrow \Gpos$, the {\em energy} of $u$ is
\begin{align} \label{Edef}
E[u] = \frac{1}{2} \int_a^b  \big\|  \frac{\partial u}{\partial t} \big\|_{u}^2 dt = \frac{1}{2} \int_a^b \int_M \big( \frac{\partial u}{\partial t} \big)^2 K_u dA_u dt,
\end{align}
where
\begin{align} \label{nsq}
\big\|  \frac{\partial u}{\partial t}  \big\|_{u}^2 = \big\llangle \frac{\partial u}{\partial t}, \frac{\partial u}{\partial t} \big\rrangle_{u}
\end{align}
The {\em energy density} is
\begin{align} \label{EDef}
E_u(t) = \big\|  \frac{\partial u}{\partial t} \big\|_{u}^2 = \int_M \big( \frac{\partial u}{\partial t} \big)^2 K_u dA_u.
\end{align}
The {\em length} of $u$ is
\begin{align} \label{Ldef}
L[u] = \int_a^b  \big\| \frac{\partial u}{\partial t} \big\|_{u} dt = \int_a^b \Big[ \int_M \big( \frac{\partial u}{\partial t} \big)^2 K_u dA_u \Big]^{\frac{1}{2}} dt.
\end{align}
\end{defn}

By taking the first variation of the energy we arrive at the geodesic equation:

\begin{lemma} $u : [a,b] \rightarrow \Gpos$ is a geodesic if and only if
\begin{align} \label{geodesic}
0 = \frac{D}{\partial t}\frac{\partial u}{\partial t} = \frac{\partial^2 u}{\partial t^2} + \frac{1}{K_u} |\nabla_u \frac{\partial u}{\partial t}|^2.
\end{align}
\end{lemma}

Suppose $u : [a,b] \rightarrow \Gpos$ is a geodesic, and write $g_u = e^{2u}g_0$, where $g_0 \in \Gpos$.  By the Gauss curvature equation,
\begin{align} \label{GKE1}
K_u = e^{-2u} \big( K_0 - \Delta_0 u),
\end{align}
where $K_0$ is the Gauss curvature of $g_0$.  Therefore, we can rewrite (\ref{geodesic}) as
\begin{align} \label{geodesicpde}
\frac{\partial^2 u}{\partial t^2} + \frac{ \displaystyle |\nabla_0 \frac{\partial u}{\partial t}|^2}{\displaystyle K_0 - \Delta_0 u} = 0.
\end{align}
As we will see in Section \ref{geoexist}, this is a degenerate elliptic fully nonlinear PDE.

In the next lemma we show two basic properties of geodesics.  As a preface, we remark that there is a canonical isometric
splitting of $T\Gpos$ with respect to
the metric.  In particular, the real line $\mathbb R \subset T_u \Gpos$
given by constant functions is orthogonal to
\begin{align*}
T^0_u \Gpos := \left\{ \ga\ |\ \int_M \alpha K_u dA_u = 0 \right\}.
\end{align*}
We will see that geodesics preserve this isometric splitting, and are automatically parameterized
with constant speed:

\begin{lemma} \label{metricsplit} Let $\phi \in C^1(\mathbb{R})$, and $u : [a,b] \rightarrow \Gpos$ a geodesic.  Then
\begin{align} \label{pconst}
\frac{d}{dt} \int_M \phi\big(\frac{\partial u}{\partial t}\big) K_u dA_u = 0.
\end{align}
In particular,
\begin{align} \label{cspeed} \begin{split}
\frac{d}{dt} \int_M \frac{\partial u}{\partial t} K_u dA_u =&\ 0,\\
\frac{d}{dt} \int_M  (\frac{\partial u}{\partial t})^2 K_u dA_u  =&\ 0.
\end{split}
\end{align}
\begin{proof} Differentiating, integrating by parts, and using the geodesic equation gives
\begin{align*}
\frac{d}{dt} \int_M \phi( \frac{\partial u}{\partial t}) K_u dA_u  =&\ \int_M \Big\{  \ddt \big[ \phi (\frac{\partial u}{\partial t}) \big] K_u dA_u + \phi( \frac{\partial u}{\partial t}) \ddt(K_u dA_u) \Big\} \\
=&\ \int_M \Big\{ \phi'\big(\frac{\partial u}{\partial t}\big) \frac{\partial^2 u}{\partial t^2} K_u - \phi\big(\frac{\partial u}{\partial t}\big) \Delta( \frac{\partial u}{\partial t} ) \Big\} dA_u \\
=&\ \int_M \Big\{  \phi'\big(\frac{\partial u}{\partial t}\big) \frac{\partial^2 u}{\partial t^2} K_u  +  \phi'( \frac{\partial u}{\partial t}) \big| \nabla_u \frac{\partial u}{\partial t} \big|_u^2  \Big\} dA_u \\
=&\ \int_M  \phi'\big(\frac{\partial u}{\partial t}\big) \Big\{  \frac{\partial^2 u}{\partial t^2}  + \frac{1}{K_u} \big| \nabla_u \frac{\partial u}{\partial t} \big|_u^2  \Big\} K_u dA_u \\
=&\ 0.
\end{align*}
\end{proof}
\end{lemma}

Choosing
\begin{align*}
\phi(t) =
\begin{cases}
t^p & t \geq 0 \\
0 & t < 0,\\
\end{cases}
\end{align*}
with $p >> 1$ large and apply (\ref{pconst}), then in the limit as $p \rightarrow \infty$ we have the following corollary of Lemma \ref{metricsplit}:  \\

\begin{cor} \label{supcor} If $u : [a,b] \rightarrow \Gpos$ is a geodesic, then $\sup_M \frac{\partial u}{\partial t}$ and $\inf_M \frac{\partial u}{\partial t}$
are constant in time.
\end{cor}

\subsection{Example: The round sphere}

Let $(S^2,g_0)$ denote the round sphere.  Using stereographic projection $\sigma : S^2 \setminus \{ N \} \rightarrow \mathbb{R}^2$,
where $N \in S^2$ denotes the north pole, one can define a one-parameter of conformal maps of $S^2$ by conjugating the dilation map $\delta_{\alpha} : x \mapsto \alpha^{-1} x$ on the plane with $\sigma$:
\begin{align*}
\varphi_{\alpha} = \sigma^{-1} \circ \delta_{\alpha} \circ \sigma : S^2 \rightarrow S^2.
\end{align*}
Taking $\alpha(t) = e^{\lambda t}$, where $\lambda$ is a fixed real number, we can define the path of conformal metrics
\begin{align} \label{gtdef}
g(t) = e^{2u} g_{0} = \phi_{\ga}^* g_0 = \left[ \frac{2 \ga(t) }{(1 + \xi) + \alpha(t)^2 (1 - \xi)} \right]^{2},
 \end{align}
where $\xi = x^3$ is the coordinate function (see \cite{LeeParker}).

\begin{prop} \label{conformalpathsgeodesics} The path $u :  (-\infty , + \infty) \rightarrow \Gpos$ is a geodesic.  \end{prop}

\begin{proof}  By (\ref{gtdef}),
 \begin{align*}
  u = \log 2 \ga - \log \left[ (1 + \xi) + \ga^2 (1 - \xi) \right].
 \end{align*}
Letting subscripts denote differentiation in $t$, we have
\begin{align*}
u_t = \frac{\alpha_t}{\ga} - \frac{2 \ga \alpha_t (1 - \xi)}{(1 + \xi) +
\ga^2(1 - \xi)},
\end{align*}
and hence
\begin{align*}
u_{tt} =&\ \frac{\alpha_{tt}}{\ga} - \left( \frac{\alpha_t}{\ga} \right)^2 -
\frac{[(1 + \xi) + \ga^2(1 - \xi)] (2 \ga \alpha_{tt} + 2 \alpha_t^2)(1-\xi) - 4
\ga^2 \alpha_t^2 (1-\xi)^2}{\left[ (1 + \xi) + \ga^2(1-\xi) \right]^2}.
\end{align*}
Since $\alpha = e^{\lambda t}$, this simplifies to
\begin{align} \label{accu}
u_{tt} = \frac{ - 4 \lambda^2 \alpha^2 (1-\xi)^2 }{\left[ (1 + \xi) + \ga^2(1-\xi) \right]^2}.
\end{align}
Also, if $\N$ denotes the connection with respect to the round metric,
\begin{align} \label{rhsu} \begin{split}
\N u_t =&\ \frac{2 \ga \alpha_t \N \xi}{(1 + \xi) + \ga^2(1-\xi)} +
\frac{2 \ga \alpha_t(1 - \xi)}{ \left[ (1 + \xi) + \ga^2(1-\xi) \right]^2}
\left[ (1-\ga^2) \N \xi \right]\\
=&\ \frac{2 \ga \alpha_t \N \xi}{\left[ (1 + \xi) + \ga^2(1-\xi) \right]^2}
\left[ (1+ \xi) + \ga^2(1-\xi) + (1-\xi)(1 - \ga^2) \right]\\
=&\ \frac{ 4\ga \alpha_t \N \xi}{\left[ (1 + \xi) + \ga^2(1-\xi) \right]^2}.
\end{split}
\end{align}
Using the fact that $\xi$ satisfies
\begin{align} \label{xid}
|\N \xi|^2 = 1 - \xi^2,
\end{align}
it follows from (\ref{rhsu}) and (\ref{xid}) that
\begin{align} \label{rhsK}
|\N_u u_t|^2 = \frac{ 4 \lambda^2 \alpha^2 (1-\xi)^2 }{\left[ (1 + \xi) + \ga^2(1-\xi) \right]^2}.
\end{align}
Since $K_u = 1$ for all $t$, comparing (\ref{accu}) and (\ref{rhsK}) we see that $u$ satisfies the geodesic equation (\ref{geodesic}).
\end{proof}

\subsection{Geodesics in the negative cone}

Analogous to the definitions for the positive cone, the energy and length of a path $w : [a,b] \rightarrow \Gneg$ are

\begin{align} \label{Enegdef}
E[w] = \frac{1}{2} \int_a^b  \big\llangle \frac{\partial w}{\partial t}, \frac{\partial w}{\partial t} \big\rrangle_{w} dt = \frac{1}{2} \int_a^b \int_M \big( \frac{\partial w}{\partial t} \big)^2 (-K_w) dA_w dt,
\end{align}
\begin{align} \label{Lnegdef}
L[w] = \int_a^b  \big\llangle \frac{\partial w}{\partial t}, \frac{\partial w}{\partial t} \big\rrangle_{w}^{\frac{1}{2}} dt = \int_a^b \Big[ \int_M \big( \frac{\partial w}{\partial t} \big)^2 (-K_w) dA_w \Big]^{\frac{1}{2}} dt.
\end{align}

By taking the first variation of $E$ we arrive at the geodesic equation for the negative cone:

\begin{lemma} \label{geodefneg} $w : [a,b] \rightarrow \Gneg$ is a geodesic if and only if
\begin{align} \label{geodesicNeg}
0 = \frac{D}{\partial t}\frac{\partial w}{\partial t} = \frac{\partial^2 w}{\partial t^2} + \frac{1}{K_w} |\nabla_w \frac{\partial w}{\partial t}|^2.
\end{align}
\end{lemma}

Equation (\ref{geodesicNeg}) is equivalent to the PDE
\begin{align} \label{geodesicNegpde}
\frac{\partial^2 w}{\partial t^2} + \frac{ \displaystyle |\nabla_0 \frac{\partial w}{\partial t}|^2}{\displaystyle K_0 - \Delta_0 w } = 0.
\end{align}
We also have the basic properties of Lemma \ref{metricsplit}:

\begin{lemma} \label{metricsplitNeg} Let $w : [a,b] \rightarrow \Gneg$ be a geodesic.  Then
\begin{align*}
\frac{d}{dt} \int_M \frac{\partial w}{\partial t} (-K_w) dA_w =&\ 0,\\
\frac{d}{dt} \int_M  (\frac{\partial w}{\partial t})^2 (-K_w) dA_w  =&\ 0.
\end{align*}
\end{lemma}

\section{Existence of geodesics}  \label{geoexist}

In this section we prove {\em a priori} estimates for solutions of the geodesic equation in the positive and negative cones.  We begin by introducing a regularization of
the geodesic equation, then prove estimates for derivatives up to order two which are independent of the regularizing parameter.  Using a continuity argument, we show that
classical solutions of the regularized equation exist and are unique.  This allows us to define the notation of a {\em regularizable geodesic} (see Definition \ref{RG}).

We begin with estimates for a geodesic $u$ in the positive cone.  By Lemma \ref{metricsplit}, we may assume $u : [0,1] \rightarrow \Gpos$.  Recall from (\ref{geodesicpde}) that
$u$ satisfies the PDE
\begin{align} \label{georecap}
\frac{\partial^2 u}{\partial t^2} + \frac{ \displaystyle |\nabla_0 \frac{\partial u}{\partial t}|^2}{\displaystyle K_0 - \Delta_0 u} = 0,
\end{align}
with boundary conditions
\begin{align} \label{bdycon}
u(\cdot,0) = u_0, \quad \quad u(\cdot,1) = u_1,
\end{align}
where $u_0, u_1 \in \Gpos$.
To simplify the notation, in the following we will
 denote derivatives with respect to $t$ by subscripts, and we will omit the subscript $0$: all metric-dependent quantities are with respect to the background
 metric $g_0$.

As we will see, (\ref{georecap}) is degenerate elliptic, so it will be necessary to regularize the equation.  To simplify some of the estimates, and to clarify the
dependence on the boundary data and other parameters, we will choose a fairly specific regularization.

Define the operator
\begin{align} \label{G1}
\mathcal{G}_f(u) = u_{tt} (-\Delta u + K) + |\nabla u_t|^2 + f.
\end{align}
Suppose $u$ is a solution of
\begin{align} \label{Gf}
\mathcal{G}_f(u) = 0
\end{align}
with $f =0$.  If in addition $u$ is {\em admissible}, i.e.,
\begin{align} \label{pca}
-\Delta u + K > 0,
\end{align}
then $u \in \Gpos$ and solves (\ref{georecap}).  In Lemma \ref{ellipticLemma} we will see that (\ref{Gf}) is elliptic when $f > 0$, but degenerate elliptic when $f = 0$.
Therefore, to prove the existence of solutions of the geodesic equation, we proceed as follows: \\

\begin{enumerate}

\item Given admissible boundary data $u_0$, $u_1$, we prove the existence of an admissible solution $\tilde{u}$ to (\ref{Gf}), for a specific choice of $f = f_0 > 0$.  \\

\item For $0 < \epsilon \leq 1$ we establish {\em a priori} estimates for solutions of
\begin{align*} 
\mathcal{G}_{f}(u) = 0 \quad \quad \quad (\star_{\epsilon})
\end{align*}
with $f = \epsilon f_0$, and subject to the given boundary conditions.  \\

\item Linearizing equation $(\star_{\epsilon})$, we show that the set of $\epsilon \in [\epsilon_0, 1]$ such that $(\star_{\epsilon})$
admits a solution satisfying the given boundary data is open, for any fixed $0 < \epsilon_0 < 1$.  Combining this with the {\em a priori} estimates of
part $(2)$, the existence of solutions for all $0 < \epsilon \leq 1$ follows.  \\

\item Taking the limit as $\epsilon \rightarrow 0$, we obtain a solution $u$ of the geodesic equation (\ref{georecap}).

\end{enumerate}
\vskip.2in

For the first step, let
\begin{align} \label{tu}
\tilde{u} = \tilde{u}(x,t) = (1-t)u_0 + t u_1 + A t (1-t),
\end{align}
where $A > 0$ will be specified later.  We easily calculate
\begin{align} \label{uparts} \begin{split}
\tilde{u}_{tt} &= - 2A, \\
-\Delta \tilde{u} + K &= (1-t)(-\Delta u_0 + K) + t (-\Delta u_1 + K), \\
|\nabla \tilde{u}_t|^2 &= |\nabla (u_1 - u_0)|^2.  \\
\end{split}
\end{align}
Since $u_0$ and $u_1$ are admissible, there is a $\delta_0 > 0$ such that
\begin{align} \label{u01ad}
(-\Delta u_0 + K) \geq \delta_0, \quad (-\Delta u_1 + K) \geq \delta_0,
\end{align}
hence
\begin{align} \label{uadm}
-\Delta \tilde{u} + K = (1-t)(-\Delta u_0 + K) + t (-\Delta u_1 + K) \geq \delta_0 > 0,
\end{align}
and it follows that $\tilde{u}$ is admissible. By (\ref{uparts}),
\begin{align} \label{Gtu}
\tilde{u}_{tt} (-\Delta \tilde{u} + K ) + |\nabla \tilde{u}_t|^2 = -2A \big[ (1-t)(-\Delta u_0 + K) + t (-\Delta u_1 + K) \big] + |\nabla (u_1 - u_0)|^2 \equiv -f_0,
\end{align}
Choosing $A = A_0 > 0$ large enough (depending only on the boundary data), we have
\begin{align} \label{fz} \begin{split}
f_0 &= 2A_0 \big[ (1-t)(-\Delta u_0 + K) + t (-\Delta u_1 + K) \big] - |\nabla (u_1 - u_0)|^2 \\
&\geq 2A \delta_0 - |\nabla (u_1 - u_0)|^2 \\
&> 0.
\end{split}
\end{align}
Therefore, $\tilde{u}$ satisfies
\begin{align} \label{Gtu2}
\tilde{u}_{tt} (-\Delta \tilde{u} + K ) + |\nabla \tilde{u}_t|^2 + f_0 = 0,
\end{align}
with $f_0 > 0$.  We have thus proved  \\

\begin{lemma} \label{initLemma}  Given admissible boundary data $u_0, u_1$, there is a function $f_0 >0$ and an admissible solution $u = \tilde{u}$ of (\ref{Gf})
with $f = f_0$.
\end{lemma}

The next step is to prove $C^2$-estimates for admissible solutions $u$ of $(\star_{\epsilon})$, where $f = \epsilon f_0$, and
$u$ satisfies the given (admissible) boundary data.  To this end, let $\mathcal{L}$ denote the linearized operator, defined by
\begin{align} \label{linop} \begin{split}
\mathcal{L}_u \phi &= \frac{d}{ds} \mathcal{G}_f(u + s \phi) \Big|_{s=0} \\
&= \phi_{tt} (-\Delta u + K)  - u_{tt} \Delta \phi   + 2  \langle \nabla u_t , \nabla \phi_t \rangle.
\end{split}
\end{align}

\begin{lemma} \label{ellipticLemma} If $u$ is an admissible solution of $(\star_{\epsilon})$ with $f > 0$, then $\mathcal{L}_u$ is
elliptic.
\end{lemma}

\begin{proof}    Fix a point $(x_0,t_0)$ and in normal coordinates at this point let $\alpha = \nabla u_t(x_0,t_0)$ and $\kappa = (-\Delta u + K)(x_0,t_0)$.  Then the symbol of $\mathcal{L}$ is given by
\begin{align*}
\sigma_{\mathcal{L}}(\xi,\tau) \phi &= \left\{  \kappa \tau^2 - u_{tt} |\xi|^2 + 2 \tau \alpha \cdot \xi \right\} \phi.
\end{align*}
Using the equation, we can substitute for $u_{tt}$ and write
\begin{align*}
\sigma_{\mathcal{L}}(\xi,\tau) \phi = \left\{ \kappa \tau^2 + \frac{f}{\kappa} |\xi|^2 + \frac{|\alpha|^2}{\kappa} |\xi|^2 + 2 \tau \alpha \cdot \xi \right\}\phi.
\end{align*}
If $f > 0$ it is clear that $\sigma_{\mathcal{L}}(\xi,\tau)\phi = 0$ if and only if $(\xi,\tau) = (0,0)$.
\end{proof}
\subsection{Derived Equations}
Before proceeding to the estimates we begin with some preliminary calculations.  We emphasize that in this subsection we will not make use of the fact
that $f = \epsilon f_0$; we will only need that $f > 0$.   \\

\begin{lemma} \label{dereqlemzed} Let $u$ be an admissible solution of $(\star_{\epsilon})$, and define
\begin{align} \label{Pdef}
P(t) = t(1-t).
\end{align}
Then
\begin{align*}
\mathcal{L}_u P = - 2 (-\Delta u + K).
\end{align*}

Also, if $\ell(t) = t$, then
\begin{align*}
\mathcal{L}_u \ell = 0.
\end{align*}
\end{lemma}

\begin{proof}  This is a straightforward calculation.  \end{proof}

\begin{lemma} \label{dereqlem1} If $u$ an admissible
solution to $(\star_{\epsilon})$, one has
\begin{align*}
 \mathcal{L}_u u &= - 2 f - K u_{tt} \\
 &= - 2 f + K \frac{ \brs{\N u_t}^2 + f}{(-\gD u + K)}.
\end{align*}
\begin{proof} We compute
\begin{align} \label{Luu}\begin{split}
 \mathcal{L}_u u &=  u_{tt}(-\Delta u + K) - u_{tt} \Delta u  +  2 \brs{\N u_t}^2 \\
&= u_{tt}(-\Delta u + K)   + u_{tt} (-\Delta u + K - K) +  2 \brs{\N u_t}^2 \\
&= - 2 f - K u_{tt}.
\end{split}
\end{align}
Since $u$ solves $(\star_{\epsilon})$ we have
\begin{align*}
- u_{tt} = \frac{|\N u_t|^2 + f}{(-\Delta u + K)},
\end{align*}
and substituting this into (\ref{Luu}) gives the result.
\end{proof}
\end{lemma}

\begin{lemma} \label{dereqlem25} If $u$ is an admissible
solution to $(\star_{\epsilon})$, one has
\begin{align*}
\mathcal{L}_u \brs{\N u}^2 \geq  \left( 2 K|\N u|^2 - 2 \langle \N K, \N u \rangle \right) \frac{|\N u_t|^2 + f}{(-\Delta u + K)} - 2 \langle \N u , \N f \rangle.
\end{align*}

\begin{proof} Differentiating the equation in the $i^{th}$-coordinate direction gives
\begin{align} \label{dig} \begin{split}
0 =&\ \N_i \mathcal{G}_f(u)\\
=&\ \N_i u_{tt} (-\Delta u + K) + u_{tt} ( -\N_i(\Delta u) + \N_i K ) + 2 \N_k u_t \N_i \N_k u_{t} + \N_i f.
\end{split}
\end{align}
Then
\begin{align} \label{LDu} \begin{split}
 \mathcal{L}_u \brs{\N u}^2 =&\ \left( 2 \langle \N u_{tt}, \N u \rangle + 2 |\N u_{t}|^2 \right) ( -\Delta u + K) - u_{tt} \left(2 |\nabla^2 u|^2 + 2 K|\N u|^2 + 2 \langle \N \Delta u, \N u \rangle
\right) \\
&\quad + 2 \N_k u_t \left( 2 \N_k \N_i u_t \N_i u  + 2 \N_k \N_i u  \N_i u_t \right)  \\
=&\ 2 \IP{\N \mathcal{G}_f(u), \N u} + 2 |\N u_{t}|^2  ( -\Delta u + K)  - u_{tt} \left(2 |\nabla^2 u|^2 + 2 K|\N u|^2 - 2 \langle \N K, \N u \rangle \right) \\
&\quad + 4 \N_k \N_i u \N_k u_t \N_i u_t - 2 \langle \N u , \N f \rangle \\
=&\ 2 |\N u_{t}|^2  ( -\Delta u + K) + (-u_{tt}) \left(2 |\nabla^2 u|^2 + 2 K|\N u|^2 - 2 \langle \N K, \N u \rangle \right)  \\
&\quad + 4 \N_k \N_i u \N_k u_t \N_i u_t - 2 \langle \N u , \N f \rangle.
\end{split}
\end{align}
Using $(\star_{\epsilon})$ then rearranging terms,
\begin{align*}
\mathcal{L}_u \brs{\N u}^2  =&\ 2 |\N u_{t}|^2    ( -\Delta u + K) +  \left(2 |\nabla^2 u|^2 + 2 K|\N u|^2 - 2 \langle \N K, \N u \rangle \right) \frac{|\N u_t|^2 + f}{(-\Delta u + K)} \\
&\quad + 4 \N_k \N_i u \N_k u_t \N_i u_t - 2 \langle \N u , \N f \rangle \\
&= 2 (|\N u_t|^2 + f)  \frac{ |\nabla^2 u|^2}{(-\Delta u + K)}  + 2 |\N u_{t}|^2  ( -\Delta u + K) + 4 \N_k \N_i u \N_k u_t \N_i u_t \\
&\quad + \left( 2 K|\N u|^2 - 2 \langle \N K, \N u \rangle \right) \frac{|\N u_t|^2 + f}{(-\Delta u + K)} - 2 \langle \N u , \N f \rangle \\
&\geq 2 |\N u_t|^2 \left|  \frac{ \N^2 u}{\sqrt{-\gD u + K}} + \frac{\N u_t \otimes \N u_t }{|\N u_t|^2}\sqrt{-\gD u + K} \right|^2 \\
&\quad + \left( 2 K|\N u|^2 - 2 \langle \N K, \N u \rangle \right) \frac{|\N u_t|^2 + f}{(-\Delta u + K)} - 2 \langle \N u , \N f \rangle \\
&\geq \left( 2 K|\N u|^2 - 2 \langle \N K, \N u \rangle \right) \frac{|\N u_t|^2 + f}{(-\Delta u + K)} - 2 \langle \N u , \N f \rangle.
\end{align*}

\end{proof}
\end{lemma}

\begin{lemma} \label{dereqlem3} If $u$ is an admissible
solution to $(\star_{\epsilon})$, then
\begin{align} \label{Ldel} \begin{split}
 \mathcal{L}_u \gD u =&\ 2 u_{tt} \frac{ |\N (-\Delta u + K)|^2}{(-\Delta u + K)} + 4 \frac{ \N_i \N_k u_t \N_k u_t \N_i (-\Delta u + K) }{(-\Delta u + K)} + 2 \frac{\langle \N f, \N (-\Delta u + K) \rangle }{(-\Delta u + K)} \\
&\ \ \ - u_{tt} \Delta K - 2 |\N^2 u_t|^2 - 2 K |\N u_t|^2 - \Delta f.
\end{split}
\end{align}
\begin{proof}
Using (\ref{dig}) we take the Laplacian of the equation to yield
\begin{align} \label{DelG} \begin{split}
0 =&\ \gD \mathcal{G}_f(u)\\
=&\ \N_i \left[   \N_i u_{tt} (-\Delta u + K) + u_{tt} ( -\N_i(\Delta u) + \N_i K ) + 2 \N_k u_t \N_i \N_k u_{t} + \N_i f \right]\\
=&\ \Delta u_{tt} (-\Delta u + K) + 2 \langle \N u_{tt} , \N (-\Delta u + K) \rangle + u_{tt}(-\Delta^2 u + \Delta K) \\
 &\ \ + 2|\N^2 u_t|^2 + 2 K |\N u_t|^2 + 2 \langle \N u, \N (\Delta u) \rangle + \Delta f \\
=&\ \mathcal{L}_u \Delta u + 2 \langle \N u_{tt} , \N (-\Delta u + K) \rangle + u_{tt} \Delta K + 2 |\N^2 u_t|^2 + 2 K |\N u_t|^2 + \Delta f.
\end{split}
\end{align}
Again using (\ref{dig}), we observe
\begin{align} \label{dig2} \begin{split}
2 \langle \N u_{tt} , \N (- \Delta u + K ) \rangle &= 2 \frac{ \langle \N \mathcal{G}_f(u), \N (-\Delta u + K) \rangle }{(-\Delta u + K)} - 2 u_{tt} \frac{ |\N (-\Delta u + K)|^2}{(-\Delta u + K)} \\
&\ \ \ - 4 \frac{ \N_i \N_k u_t \N_k u_t \N_i (-\Delta u + K) }{(-\Delta u + K)} - 2 \frac{\langle \N f, \N (-\Delta u + K) \rangle }{(-\Delta u + K)}.
\end{split}
\end{align}
Substituting this into the result of (\ref{DelG}) and rearranging terms we arrive at (\ref{Ldel}).
\end{proof}
\end{lemma}

\begin{lemma} \label{dereqlem5} If $u$ is an admissible
solution to $(\star_{\epsilon})$, one has
 \begin{align*}
  \mathcal{L}_u u_{tt} = 2 u_{ttt} \Delta u_t  - 2 |\N u_{tt}|^2 - f_{tt}.
 \end{align*}
 \begin{proof} Differentiating $(\star_{\epsilon})$ twice yields
\begin{align*}
[ \mathcal{G}_f(u)]_{tt} &=  u_{tttt} (-\Delta u + K) + 2 u_{ttt} (-\Delta u_t) + u_{tt} (-\Delta u_{tt}) + 2 \langle \N u_{ttt}, \N u_t \rangle + 2 |\N u_{tt}|^2  + f_{tt} \\
&= \mathcal{L}_u u_{tt} + 2 u_{ttt} (-\Delta u_t)  + 2 |\N u_{tt}|^2  + f_{tt},
\end{align*}
and the result follows.
 \end{proof}
\end{lemma}

\subsection{\texorpdfstring{$C^0$}{C0}-estimates}

We now turn to the estimates proper, and will frequently take advantage of the special choice of $f = \epsilon f_0$ with $0 < \epsilon \leq 1$.  In particular, we note
that $|f| \leq C \epsilon$, and
\begin{align} \label{fbounds}
\max_{M \times [0,1]} \big[ \frac{|\nabla f|^2}{f^2}  + \frac{|\nabla^2 f|}{f} \big] \leq C,
\end{align}
where $C$ only depends on the boundary data $u_0, u_1$.  In addition, since
\begin{align} \label{fo}
f_0 = 2A_0 \big[ (1-t)(-\Delta u_0 + K) + t (-\Delta u_1 + K) \big] - |\nabla (u_1 - u_0)|^2
\end{align}
is linear in $t$, it follows that
\begin{align} \label{fttz}
f_{tt} = 0.
\end{align}

\begin{prop} \label{intestprop1} If $u$ is an admissible solution
to $(\star_{\epsilon})$, then there exists a constant $C$ which depends on the boundary data and $\min_M K$, such that
\begin{align*}
\sup_{M \times [0,1]} \brs{u} \leq C.
\end{align*}
\begin{proof} First we observe that equation $(\star_{\epsilon})$ and the assumption that $f > 0$ imply that
$u_{tt} \leq 0$.  Using this and the fundamental theorem of calculus yields a
lower bound for $u$ depending on $\inf_M u_0$ and $\inf_M u_1$.

To obtain an upper bound, let $A > 0$ (to be chosen later) and define
\begin{align*}
\tilde{u} = u - A t (1-t).
\end{align*}
Then by Lemmas \ref{dereqlemzed} and \ref{dereqlem1},
\begin{align*}
\mathcal{L}_u \tilde{u} &= - 2 f + K \frac{ \brs{\N u_t}^2 + f}{(-\gD u + K)} + 2A (-\Delta u + K) \\
&\geq -2 f + 2A ( - \Delta \tilde{u} + K).
\end{align*}
If we choose $A$ large enough (depending on $\min_M K$ and $\max_{M \times [0,1]} f$), it follows from the maximum principle that $\tilde{u}$ cannot have an interior
maximum.  Therefore,
\begin{align*}
\sup_{M \times [0,1]} \tilde{u} = \sup_{M \times \{ 0,1\} } \tilde{u} = \sup_{M \times \{ 0,1 \}} u,
\end{align*}
and $u$ has an upper bound depending on $A = A(\min_M K, \max_{M \times [0,1]} f)$ and $\sup_M u_0, \sup_{M} u_1$.
\end{proof}
\end{prop}

\subsection{\texorpdfstring{$C^1$}{C1}-estimates}

\begin{prop} \label{boundaryestprop1} If $u$ is an admissible
solution of $(\star_{\epsilon})$, then there is a constant $C$ (depending on $u_0$ and $u_1$)  such that
\begin{align*}
\sup_{M \times [0,1] } \brs{u_t} \leq C.
\end{align*}

\begin{proof} Since $u_{tt} \leq 0$, it suffices to prove an upper bound for $u_t$ at $t = 0$ and a lower bound for $u_t$ at $t=1$.

To establish an upper bound at $t = 0$ we will construct a supersolution to the equation.  Let
\begin{align} \label{suprw}
w = w(x,t) = u(x,t) - u_0(x) - A t(1-t) - B t
\end{align}
where $A, B > 0$ are constants to be specified later.  Note by the definition of $\mathcal{L}_u$,
\begin{align*}
\mathcal{L}_u u_0 = \frac{|\N u_t|^2 + f}{(-\Delta u + K)} \Delta u_0.
\end{align*}
Therefore, by Lemmas \ref{dereqlemzed} and \ref{dereqlem1}
\begin{align*}
\mathcal{L}_u w = - 2f + \frac{|\N u_t|^2 + f}{(-\Delta u + K)} (K - \Delta u_0) + 2A(-\Delta u + K).
\end{align*}
Since $u_0$ is admissible,
\begin{align*}
\kappa_0 = \min_{M} (K - \Delta u_0) > 0,
\end{align*}
and it follows that
\begin{align*}
 \mathcal{L}_u w \geq - 2f + \frac{\kappa_0 f}{(-\Delta u + K)}  + 2A(-\Delta u + K).
 \end{align*}
Also,
\begin{align*}
2 \sqrt{2A \kappa_0 f} \leq  \frac{\kappa_0 f}{(-\Delta u + K)}  + 2A(-\Delta u + K),
\end{align*}
hence
\begin{align*}
 \mathcal{L}_u w &\geq - 2f + 2 \sqrt{2A \kappa_0 f}  \\
 &= 2 \sqrt{f} \big( - \sqrt{f}  +  \sqrt{2A \kappa_0} \big).
 \end{align*}
In particular, if we choose $A > 0$ large enough depending on $\max_M f$ and $\kappa_0$, then $\mathcal{L}_u w \geq 0$, and by the maximum principle $w$ cannot have an interior maximum.
Notice $w(x,0) = 0$, while
\begin{align*}
w(x,1) = u_1(x) - u_0(x) - B.
\end{align*}
Therefore, if we choose $B$ large enough (depending on $\max_M u_1$ and $\min_M u_0$), we can arrange so that $w(x,1) \leq 0$, and consequently
\begin{align*}
\max_{ M \times [0,1] } w \leq 0.
\end{align*}
Therefore, for $t > 0$,
\begin{align*}
\frac{ w(x,t) }{t} = \frac{u(x,t) - u_0(x) - A t(1-t) - Bt }{t} \leq 0,
\end{align*}
and letting $t \rightarrow 0^{+}$ we conclude
\begin{align*}
u_t(x,0) \leq (A + B).
\end{align*}

To prove a lower bound for $u_t$ at $t=1$, we define
\begin{align*}
\tilde{w}(x,t) = u_1(x) - u(x,t) + At (1-t) + Bt,
\end{align*}
and the argument proceeds as before.
\end{proof}
\end{prop}

\begin{prop} \label{intestDu} If $u$ is an admissible solution
to $(\star_{\epsilon})$, then there is a constant $C$ depending on the boundary data $u_0, u_1$ and $\min_M K$, such that
\begin{align*}
\sup_{M \times [0,1]} |\N u|^2 \leq C.
\end{align*}

\begin{proof}  If the supremum of $|\N u|$ is attained when $t = 0$ or $t=1$ then we are done.  Therefore, assume $|\N u|$ attains an interior maximum.  Let $\Lambda > 0$ (to be chosen
later) and define
\begin{align} \label{Dutest}
\psi = |\N u|^2 - \Lambda t(1-t).
\end{align}
By Lemmas \ref{dereqlemzed} and \ref{dereqlem25},
\begin{align*}
\mathcal{L}_u \psi \geq  \left( 2 K|\N u|^2 - 2 \langle \N K, \N u \rangle \right) \frac{|\N u_t|^2 + f}{(-\Delta u + K)} - 2 \langle \N u , \N f \rangle + 2 \Lambda (-\Delta u + K).
\end{align*}
We may assume that $|\N u|$ is large enough so that
\begin{align*}
 2 K|\N u|^2 - 2 \langle \N K, \N u \rangle \geq K |\N u|^2,
 \end{align*}
hence
\begin{align*}
\mathcal{L}_u \psi &\geq   K|\N u|^2  \frac{|\N u_t|^2 + f}{(-\Delta u + K)} - 2 \langle \N u , \N f \rangle + 2 \Lambda (-\Delta u + K) \\
&\geq \frac{K f |\N u|^2 }{(-\Delta u + K)} - 2 \langle \N u , \N f \rangle + 2 \Lambda (-\Delta u + K).
\end{align*}
Since
\begin{align*}
2 \langle \N f, \N u \rangle \leq  \frac{K f |\N u|^2 }{(-\Delta u + K)} + \frac{1}{K} \frac{|\N f|^2}{f} ( -\Delta u + K),
\end{align*}
we have
\begin{align*}
\mathcal{L}_u \psi &\geq  \left( 2 \Lambda - \frac{1}{K} \frac{|\N f|^2}{f} \right) (-\Delta u + K).
\end{align*}
Choosing $\Lambda >> 1$ large enough (depending on $\max_{M \times [0,1]} |\N f|^2/f$ and $\min_M K$) we can arrange so that
\begin{align*}
\mathcal{L}_u \psi \geq 0,
\end{align*}
hence by the maximum principle $\psi$ cannot have an interior max.  Since $\psi = |\N u|^2$ on $M \times \{0,1\}$, we conclude
\begin{align*}
\sup_{M \times [0,1]} \psi = \sup_{M \times \{ 0,1 \}} \psi = \sup_{M \times \{ 0,1 \}} |\N u|^2.
\end{align*}
\end{proof}
\end{prop}

\subsection{\texorpdfstring{$C^2$}{C2}-estimates}

\begin{prop} \label{intestprop3} If $u$ is an admissible solution
to $(\star_{\epsilon})$, there exists a constant $C$ such that
\begin{align*}
\sup_{M \times [0,1]} |\gD u| \leq  C.
\end{align*}
\begin{proof}  Let $Q = -\Delta u + K$; then it suffices to prove a bound for $\sup_M |Q|$.  Since $u$ is admissible, $Q > 0$ and so we only need to
prove an upper bound for $Q$. By Lemma \ref{dereqlem3},
\begin{align} \label{LQ} \begin{split}
\mathcal{L}_u Q &= - 2 u_{tt} \frac{ |\N Q |^2}{Q} - 4 \frac{ \N_i \N_k u_t \N_k u_t \N_i Q }{Q} - 2 \frac{\langle \N f, \N Q \rangle }{Q} \\
&\ \ \ + 2 |\N^2 u_t|^2 + 2 K |\N u_t|^2 + \Delta f \\
&\geq - 2 u_{tt} \frac{ |\N Q |^2}{Q} - 4 \frac{ \N_i \N_k u_t \N_k u_t \N_i Q }{Q} - 2 \frac{\langle \N f, \N Q \rangle }{Q} + \Delta f.
\end{split}
\end{align}
Let
\begin{align*}
W = Q - \Lambda t (1-t),
\end{align*}
where $\Lambda > 0$ will be chosen.  Then
\begin{align} \label{LW} \begin{split}
\mathcal{L}_u W \geq - 2 u_{tt} \frac{ |\N W |^2}{Q} - 4 \frac{ \N_i \N_k u_t \N_k u_t \N_i W }{Q} - 2 \frac{\langle \N f, \N W \rangle }{Q} + \Delta f + 2 \Lambda Q.
\end{split}
\end{align}
If $Q$ attains its maximum when $t = 0$ or $t=1$, then we are done.  Therefore, assume $Q$ has an interior maximum, and furthermore that $\max_{M \times [0,1]} Q > 1$.  Then at the
point where $Q$ attains its maximum, $\N Q = 0$ and (\ref{LW}) becomes
\begin{align*}
\mathcal{L}_u W &\geq \Delta f + 2 \Lambda Q \\
&\geq \Delta f + 2 \Lambda.
\end{align*}
However, if $\Lambda > 0$ is chosen larger than $\min_{M \times [0,1]} \Delta f$, then we have $\mathcal{L}_u W > 0$, a contradiction.
\end{proof}
\end{prop}

\begin{prop} \label{intestutt} If $u$ is an admissible solution
to $(\star_{\epsilon})$, then there is a constant $C$ such that
\begin{align*}
\sup_{M \times [0,1]} |u_{tt}| = \sup_{M \times \{ 0,1\}} |u_{tt}|.
\end{align*}
\end{prop}

\begin{proof}  Since $u_{tt} < 0$, we only need to estimate the infimum of $u_{tt}$.

For the proof we will use the special form of the function $f$.  Recall $f = \epsilon f_0$, where $f_0$ is
given in (\ref{fz}).  As we observed in (\ref{fttz}),
\begin{align*}
f_{tt} = 0.
\end{align*}
Therefore, by Lemma \ref{dereqlem5},
 \begin{align*}
  \mathcal{L}_u u_{tt} &= 2 u_{ttt} \Delta u_t  - 2 |\N u_{tt}|^2 - f_{tt} \\
 &= 2 u_{ttt} \Delta u_t  - 2 |\N u_{tt}|^2.
 \end{align*}
 It follows from the strong maximum principle that $u_{tt}$ cannot have an interior minimum, and the lemma follows.
\end{proof}

\begin{prop} \label{C2mixed} If $u$ is an admissible solution
to $(\star_{\epsilon})$, then there is a constant $C$ such that
\begin{align*}
\sup_{M \times \{ 0,1 \}} \big[ |u_{tt}| + |\nabla u_t| + |\nabla^2 u| \big] \leq C.
\end{align*}
\end{prop}

\begin{proof}  Observe that a bound for $|\nabla^2 u|$ on the boundary is immediate.  If we can prove a bound on the `mixed' term $|\nabla u_t|$, then restricting
to equation to $t = 0$ we have
\begin{align*}
u_{tt}(\cdot,0) \big( -\Delta u_0 + K \big) + |\nabla u_t(\cdot,0)|^2 + f = 0,
\end{align*}
hence
\begin{align*}
\delta_0 \sup_M \big( -u_{tt}(\cdot,0) \big) \leq \sup_M |\nabla u_t(\cdot,0)|^2 + C,
\end{align*}
where
\begin{align*}
\delta_0 = \min_M \big( -\Delta u_0 + K \big) > 0.
\end{align*}
Therefore, a bound on $|\nabla u_t|$ implies a bound on $|u_{tt}|$.

To prove a bound on $\nabla u_t$ we consider the following auxiliary function
$\Psi : M \times [0,\tau] \rightarrow \mathbb{R}$, where $0 < \tau < 1$ will be chosen later:
\begin{align} \label{Psidef}
\Psi = |\nabla(u-u_0)| + A (u-u_0) - Bt(1-t),
\end{align}
where $A,B > 0$ are to be determined.

We claim that by choosing $A,B >> 0$ large enough and $\tau > 0$ small enough, that
$\Psi$ attains a non-positive maximum on the boundary of of $M \times [0,\tau]$.  Assuming for
the moment this is true, let us see how a bound for $\N u_t$ follows.

Choose a point $x_0 \in M$, and a unit tangent vector $X \in T_{x_0}M$.  Let $\{ x^i \}$ be a local coordinate system
with $X = \frac{\partial}{\partial x^1}$ at $x_0$.  Then
\begin{align*}
&\frac{\partial}{\partial x^1}\big( u(x,t) - u_0(x) \big) + A (u(x,t) - u_0(x)) - Bt(1-t) \\
& \quad \leq  |\nabla(u-u_0)(x,t)| + A (u-u_0)(x,t) - Bt(1-t)  \\
& \quad \leq 0.
\end{align*}
Therefore,
\begin{align*}
0 &\geq \lim_{t \to 0+} \frac{1}{t} \Big\{ \frac{\partial}{\partial x^1}u(x,t) - \frac{\partial}{\partial x^1}u_0(x) + A (u(x,t) - u_0(x)) - Bt(1-t) \Big\} \\
&= \frac{\partial}{\partial x^1}u_t(x_0,0) + A u_t(x_0,0) - B.
\end{align*}
Since $u_t$ is bounded, an upper bound on $\frac{\partial}{\partial x^1}u_t$ follows.  Since $X = \frac{\partial }{\partial x^1}$ was arbitrary, we obtain a bound on $|\N u_t(x,0)|$.

To see that such a choice of $A,B,$ and $\tau$ are possible, we first note that
\begin{align} \label{Pzb}
\Psi(x,0) = 0.
\end{align}
Since $|\N u|$ is bounded,
\begin{align*}
\Psi(x,\tau) &= |\nabla  u(x,\tau) - \nabla  u_0(x)|   + A  \big( u(x,\tau) - u_0(x) \big) - B \tau (1-\tau) \\
&\leq C_1 + A \vert u(x,\tau) - u_0(x) \vert -  B \tau (1-\tau).
\end{align*}
Since $|u_t|$ is also bounded,
\begin{align*}
A \vert u(x,\tau) - u_0(x) \vert \leq C_2 A \tau,
\end{align*}
hence if $0 < \tau < 1/2$,
\begin{align*}
\Psi(x,\tau) &\leq C_1 + C_2 A \tau -  B \tau (1-\tau) \\
&\leq C_1 + (C_2 A -  \frac{B}{2} ) \tau.
\end{align*}
Therefore, if $B$ is chosen large enough (depending on $\tau$, $A$, $C_1$, and $C_2$), then
\begin{align} \label{Ptb}
\Psi(x,\tau) \leq 0.
\end{align}
We conclude that $\Psi \leq 0$ on $\partial \big( M \times [0,\tau] \big)$.

Assume the maximum of $\Psi$ is attained at a point $(x_0,t_0)$ which is interior (i.e., $0 < t_0 < \tau$).  Let
\begin{align} \label{etadef}
\eta = \frac{\displaystyle \N (u-u_0)(x_0,t_0)}{ \displaystyle | \N (u-u_0)(x_0,t_0)|}.
\end{align}
We can extend $\eta$ locally via parallel transport along radial geodesics based at $x_0$. By construction,
\begin{align} \label{pareta} \begin{split}
\N \eta (x_0) &= 0, \\
|\N^2 \eta (x_0)| &\leq C(g).
\end{split}
\end{align}
By using a cut-off function, we can assume $\eta$ is globally defined and satisfies
\begin{align*}
|\eta| \leq 1,
\end{align*}
with $|\eta| = 1$ in a neighborhood of $x_0$.

Define
\begin{align} \label{Hdef}
H = \ea \N_{\alpha} (u-u_0) + A(u-u_0) - Bt(1-t).
\end{align}
Since $|\eta| \leq 1$,
\begin{align} \label{Hless}
H(x,t) \leq \Psi(x,t),
\end{align}
and the max of $H$ is attained at $(x_0,t_0)$.  Therefore,
\begin{align} \label{MP1}
\mathcal{L}_u H (x_0,t_0) \leq 0.
\end{align}

To compute $\mathcal{L}_u H (x_0,t_0)$, let $\phi = \ea \N_{\alpha} (u-u_0).$  Using (\ref{pareta}), at $(x_0,t_0)$ we have
\begin{align} \label{phis} \begin{split}
\phi_t &= \ea \N_{\alpha} u_t, \\
\phi_{tt} &= \ea \N_{\alpha} u_{tt}, \\
\nabla_k \phi_t &= \ea \N_k \N_{\alpha} u_t.
\end{split}
\end{align}
Also, commuting derivatives and using the fact that $\N \eta(x_0) = 0$, we have
\begin{align*}
\Delta \phi &= \ea \Delta \N_{\alpha} (u - u_0)  + \Delta \ea \N_{\alpha}(u-u_0) \\
&= \ea \N_{\alpha} (\Delta u) - \ea  \N_{\alpha} (\Delta u_0) + K \ea \N_{\alpha} (u-u_0) + \Delta \ea \N_{\alpha} (u-u_0) \\
&\geq \ea \N_{\alpha} (\Delta u) -  C,
\end{align*}
where in the last line we have used (\ref{pareta}) and the fact that $|\N u|$ is bounded.
Combining the above, we have
\begin{align} \label{Lp1} \begin{split}
\mathcal{L}_u \phi &= \phi_{tt} (-\Delta u + K) - u_{tt} \Delta \phi + 2 \langle \N u_t , \N \phi_t \rangle \\
&\geq \ea \N_{\alpha} u_{tt} (-\Delta u + K) - u_{tt} \big\{ \ea \N_{\alpha} (\Delta u) - C \big\}  \\
& \quad + 2 \ea \N_k u_t  \N_k \N_{\alpha} u_t.
\end{split}
\end{align}
Also, differentiating the equation in the direction of $\eta$,
\begin{align*}
0 &= \ea \N_{\alpha} \big\{ u_{tt} (-\Delta u + K) + |\N u_t|^2 + f \big\} \\
&= \ea \N_{\alpha} u_{tt} (-\Delta u + K) - u_{tt} \ea \N_{\alpha} (\Delta u) + u_{tt} \ea \N_{\alpha} K + 2 \ea \N_{\alpha} \N_k u_t \N_k u_t + \ea \N_{\alpha} f.
\end{align*}
Combining this with (\ref{Lp1}), we get
\begin{align} \label{Lp} \begin{split}
\mathcal{L}_u \phi &\geq (-u_{tt}) \big\{ - \ea \N_{\alpha} K- C \big\} + \ea \N_{\alpha} f \\
&\geq - C_1 (-u_{tt}) - C_2 |\N f|.
\end{split}
\end{align}
By Lemmas \ref{dereqlem1} and \ref{dereqlemzed} we conclude
\begin{align} \label{LH}
\mathcal{L}_u H \geq (-u_{tt}) \big\{ - C_1 + A (-\Delta u_0 + K) \big\} - C_2 |\N f| - 2 A f + 2 B (-\Delta u + K).
\end{align}
Let
\begin{align*}
\delta_0 = \min_M (- \Delta u_0 + K) > 0.
\end{align*}
Then if $A > 0$ is chosen large enough, the term in braces in (\ref{LH}) is bounded below by $(A/2)\delta_0$.  Also, by (\ref{fbounds}),
\begin{align*}
- C_2 |\N f| - 2 A f \geq - 3 A f,
\end{align*}
again if $A$ is large enough.  Therefore,
\begin{align} \label{LH2}
\mathcal{L}_u H \geq \frac{1}{2} A \delta_0 (-u_{tt}) -  3 A f + 2 B (-\Delta u + K).
\end{align}
By the equation,
\begin{align*}
-u_{tt} &= \frac{|\N u_t|^2 + f}{-\Delta u + K},
\end{align*}
hence
\begin{align*}
\mathcal{L}_u H &\geq \frac{1}{2} A \delta_0 \frac{|\N u_t|^2 + f}{-\Delta u + K} - 3 A f + 2 B (-\Delta u + K).
\end{align*}
Using the arithmetic-geometric mean inequality,
\begin{align} \label{AGM}
2 \big( \delta_0 AB\big)^{1/2} \sqrt{ |\N u_t|^2 + f} \leq \frac{1}{2} A \delta_0 \frac{|\N u_t|^2 + f}{-\Delta u + K} + 2 B (-\Delta u + K).
\end{align}
It follows that
\begin{align*}
0 \geq \mathcal{L}_u H &\geq 2 \big( \delta_0 AB\big)^{1/2} \sqrt{ |\N u_t|^2 + f} - 3 A f \\
&\geq 2 \big( \delta_0 AB\big)^{1/2} \sqrt{ f} - 3A f.
\end{align*}
We can assume $B$ is large enough so that $\delta_0 B > 1$.  Therefore,
\begin{align} \label{bvs}
0 \geq 2 \sqrt{ Af} - 3 A f.
\end{align}
Once $A$ is fixed, since $f = \epsilon f_0$, for $\epsilon > 0$ small enough we have $Af << 1$, hence $\sqrt{Af} > 2 Af$, which
contradicts (\ref{bvs}).  It follows that $H$ (hence $\Psi$) cannot have an interior maximum.
\end{proof}

\begin{cor} \label{uttMbound} If $u$ is an admissible solution
to $(\star_{\epsilon})$, there exists a constant $C$ such that
\begin{align*}
\sup_{M \times [0,1]} |u_{tt}| \leq  C.
\end{align*}
\end{cor}

\begin{cor} \label{DutMbound} If $u$ is an admissible solution
to $(\star_{\epsilon})$, there exists a constant $C$ such that
\begin{align*}
\sup_{M \times [0,1]} |\N u_t| \leq  C.
\end{align*}
\end{cor}

\begin{proof}  From the equation,
\begin{align} \label{equt}
|\N u_t|^2 = - u_{tt} (-\Delta u + K ) - f.
\end{align}
By Propositions \ref{intestprop3} and \ref{C2mixed}, the right-hand side is uniformly bounded, and a
bound for $|\N u_t|$ follows.
\end{proof}

Next, we give an estimate for the Hessian for solutions to $(\star_{\epsilon})$.

\begin{prop} \label{C2Hess} If $u$ is an admissible solution
to $(\star_{\epsilon})$, then there is a constant $C$ depending on the initial data of $u$ such that
\begin{align*}
\sup_{M \times [0,1]} |\nabla^2 u|  \leq C
\end{align*}
\end{prop}

\begin{proof}
By Proposition \ref{intestprop3} the trace of $\N^2 u$ is controlled, and therefore it suffices to control either the smallest or largest eignevalue; we will
estimate the smallest.

Since $|\N^2 u|$ is controlled on the boundary $M \times \{ 0,1\}$, we only need an interior estimate.  To this end, let
\begin{align} \label{lmin}
\lambda_{min}(x,t) = \min_{g_x(X,X) = 1} \N^2 u(x,t)(X,X)
\end{align}
denote the smallest eigenvalue of $\N^2 u$ at $(x,t)$.  Define
\begin{align} \label{Hdef2}
H = H(x,t) = \lambda_{min}(x,t) -  |\N u(x,t)|^2  + 2 t (1-t),
\end{align}
where $A, B > 0$ will be specified later, but will only depend on the boundary data $u_0, u_1$.  If $H$ attains its minimum on the boundary $M \times \{ 0,1\}$,
then we are done.  Therefore, assume $H$ attains its minimum at a point $(x_0,t_0)$, with $0 < t_0 < 1$.  Let $\eta \in T_{x_0}M$ be a unit eigenvector corresponding to the smallest
eigenvalue $\lambda_{min}(x_0,t_0)$.  As in the proof of Proposition \ref{C2mixed}, we extend $\eta$ locally by parallel transport along radial geodesics originating at $x_0$, hence $\eta$ satisfies (\ref{pareta}).  By introducing a cutoff, we can assume that $\eta$ is globally defined (in space),
with $|\eta| \leq 1$ on $M$ and $|\eta| = 1$ near $x_0$.

Now let
\begin{align} \label{Phidef}
\Phi = \Phi(x,t) = \eta^i \eta^j \N_i \N_j u -  |\N u|^2  + 2  t(1-t).
\end{align}
Then at each point $(x,t)$ near $(x_0,t_0)$,
\begin{align*}
\Phi(x,t) &= \eta^i \eta^j \N_i \N_j u -   |\N u|^2 + 2 t(1-t) \\
&\geq \lambda_{min}(x,t)|\eta|^2 -  |\N u|^2  + 2 t (1-t).
\end{align*}
Since $|\eta|^2 = 1$ in a neighborhood of $x_0$, it follows that for $x$ near $x_0$,
\begin{align*}
\Phi(x,t) \geq H(x,t),
\end{align*}
with equality when $(x,t) = (x_0,t_0)$. It follows that $\Phi$ attains a local minimum at $(x_0, t_0)$, hence
\begin{align} \label{MPPhi}
\mathcal{L}_u \Phi (x_0,t_0) \geq 0.
\end{align}

To compute $\mathcal{L}_u \Phi$, let $\phi = \eta^i \eta^j \N_i \N_j u$.  Since $\N \eta(x_0) = 0$, at $x_0$, we have
\begin{align} \label{ptx} \begin{split}
\phi_t &= \eta^i \eta^j \N_i \N_j u_t, \\
\phi_{tt} &= \eta^i \eta^j \N_i \N_j u_{tt}, \\
\N_k \phi_t &= \eta^i \eta^j \N_k \N_i \N_j u_t.
\end{split}
\end{align}
By commuting derivatives and using (\ref{pareta}), we also have (at $x_0$)
\begin{align*}
\Delta \phi &= \eta^i \eta^j \Delta (\N_i \N_j u) + \Delta \eta^i \eta^j \N_i \N_j u + \eta^i \Delta \eta^j \N_i \N_j u \\
&= \eta^i \eta^j \big\{ \N_i \N_j (\Delta u) + \N_i K \N_j u + \N_j K \N_i u - \langle \N K, \N u \rangle g_{ij} + 4 K \N_i \N_j u
- 2 K (\Delta u)g_{ij} \big\}  \\
& \quad + \Delta \eta^i \eta^j \N_i \N_j u + \eta^i \Delta \eta^j \N_i \N_j u.
\end{align*}
Since $|\N u|$ and $|\Delta u|$ are bounded, it follows
\begin{align} \label{DelP1}
\Delta \phi \leq \eta^i \eta^j  \N_i \N_j (\Delta u) + C  + C |\N^2 u|.
\end{align}
Combining (\ref{ptx}) and (\ref{DelP1}), we have
\begin{align} \label{LP1}
\mathcal{L}_u \phi \leq \eta^i \eta^j \Big\{ \N_i \N_j u_{tt} (-\Delta u + K) - u_{tt} \N_i \N_j (\Delta u)  + 2 \N_k \N_i \N_j u_t \N_k u_t \Big\}  + (-u_{tt}) \big[  C + C |\N^2 u| \big].
\end{align}

Differentiating $(\star_{\epsilon})$ once gives
\begin{align} \label{Dstr} \begin{split}
0 &= \N_j \Big\{ u_{tt} (-\Delta u + K) + |\N u_t|^2 + f \Big\} \\
&= \N_j u_{tt} (-\Delta u + K) + u_{tt} ( - \N_j(\Delta u) + \N_j K ) + 2 \N_j \N_k u_t \N_k u_t + \N_j f.
\end{split}
\end{align}
Differentiating again,
\begin{align} \label{DDstr} \begin{split}
0 &= \N_i \N_j \Big\{ u_{tt} (-\Delta u + K) + |\N u_t|^2 + f \Big\} \\
&= \N_i \N_j u_{tt} (-\Delta u + K) + \N_j u_{tt} (-\N_i(\Delta u) + \N_i K ) + \N_j u_{tt} ( - \N_i(\Delta u) + \N_i K ) \\
&\quad + u_{tt} ( -\N_i \N_j (\Delta u) + \N_i \N_j K ) + 2 \N_i \N_k u_t \N_j \N_k u_t  + 2 \N_k u_t \N_i \N_j \N_k u_t + \N_i \N_j f.
\end{split}
\end{align}
By (\ref{Dstr}), the cross terms above can be written
\begin{align} \label{crs} \begin{split}
\N_i u_{tt} & (- \N_j (\Delta u) + \N_j K) + \N_j u_{tt} (- \N_i (\Delta u) + \N_i K)   \\
& = - \frac{\N_i u_{tt} \N_j u_{tt} }{u_{tt}} (-\Delta u + K) - 2 \frac{\N_i u_{tt}}{u_{tt}} \N_j \N_k u_t \N_k u_t - \frac{\N_i u_{tt}}{u_{tt}} \N_j f  \\
&\quad - \frac{\N_j u_{tt} \N_i u_{tt} }{u_{tt}} (-\Delta u + K) - 2 \frac{\N_j u_{tt}}{u_{tt}} \N_i \N_k u_t \N_k u_t - \frac{\N_j u_{tt}}{u_{tt}} \N_i f.
\end{split}
\end{align}
Also, commuting derivatives in the second-to-last term in (\ref{DDstr}) gives
\begin{align} \label{comt}
2 \N_k u_t \N_i \N_j \N_k u_t = 2 \N_k u_t \N_k \N_i \N_j  u_t + 2 K |\N u_t|^2 g_{ij} - 2 K \N_i u_t \N_j u_t.
\end{align}
Substituting (\ref{crs}) and (\ref{comt}) into (\ref{DDstr}), pairing with $\eta^i \eta^j$, and rearranging terms gives
\begin{align} \label{LP2} \begin{split}
0 &= \eta^i \eta^j \Big\{ \N_i \N_j u_{tt} (-\Delta u + K) + u_{tt} ( -\N_i \N_j (\Delta u) + \N_i \N_j K ) +  2 \N_k u_t \N_k \N_i \N_j  u_t \\
&\quad - 2 \frac{\N_i u_{tt} \N_j u_{tt} }{u_{tt}} (-\Delta u + K) - 2 \frac{\N_i u_{tt}}{u_{tt}} \N_j \N_k u_t \N_k u_t - 2 \frac{\N_j u_{tt}}{u_{tt}} \N_i \N_k u_t \N_k u_t  \\
& \quad  + 2 \N_i \N_k u_t \N_j \N_k u_t  - \frac{\N_i u_{tt}}{u_{tt}} \N_j f  - \frac{\N_j u_{tt}}{u_{tt}} \N_i f + 2 K |\N u_t|^2 g_{ij} - 2 K \N_i u_t \N_j u_t   \\
& \quad \quad + \N_i \N_j f \Big\}.
\end{split}
\end{align}
Notice the first three terms in (\ref{LP2}) correspond to the leading terms of (\ref{LP1}).  Therefore, substituting gives
\begin{align} \label{LP3} \begin{split}
\mathcal{L}_u \phi &\leq \eta^i \eta^j \Big\{ 2 \frac{\N_i u_{tt} \N_j u_{tt} }{u_{tt}} (-\Delta u + K) + 2 \frac{\N_i u_{tt}}{u_{tt}} \N_j \N_k u_t \N_k u_t + 2 \frac{\N_j u_{tt}}{u_{tt}} \N_i \N_k u_t \N_k u_t  \\
& \quad  - 2 \N_i \N_k u_t \N_j \N_k u_t  + \frac{\N_i u_{tt}}{u_{tt}} \N_j f  + \frac{\N_j u_{tt}}{u_{tt}} \N_i f - 2 K |\N u_t|^2 g_{ij} + 2 K \N_i u_t \N_j u_t - \N_i \N_j f \Big\} \\
& \quad \quad + (- u_{tt}) \big[ C + C |\N^2 u|  \big].
\end{split}
\end{align}

Let
\begin{align*}
\omega_k &= \eta^i \N_i \N_k u_t, \\
v &= \eta^i \frac{\N_i u_{tt}}{u_{tt}}.
\end{align*}
Then (\ref{LP3}) implies
\begin{align} \label{LP4} \begin{split}
\mathcal{L}_u \phi &= 2 v^2 u_{tt} (-\Delta u + K) + 4 v \omega(\N u_t) - 2 |\omega|^2  + 2 v \langle \eta, \N f \rangle \\
& \quad + \eta^i \eta^j \big\{ - 2 K |\N u_t|^2 g_{ij} + 2 K \N_i u_t \N_j u_t - \N_i \N_j f  \big\} \\
& \quad  + (- u_{tt}) \big[  C + C |\N^2 u|  \big].
\end{split}
\end{align}
Note that using the equation, the first line in (\ref{LP4}) can be written
\begin{align} \label{in1} \begin{split}
2 v^2 u_{tt} (-\Delta u + K) + & 4 v \omega(\N u_t) - 2 |\omega|^2  + 2 v \langle |\eta|, \N f \rangle  \\
&\leq 2 v^2  u_{tt} (-\Delta u + K) + 2 v^2 |\N u_t|^2 + 2 v |\eta| |\N f| \\
&= 2 v^2 \big\{  u_{tt} (-\Delta u + K) + |\N u_t|^2 + f \big\} - 2 f v^2 + 2 v |\eta| |\N f|  \\
&= - 2 f v^2 + 2 v |\eta| |\N f| \\
&\leq \frac{1}{2} \eta^2 \frac{|\N f|^2}{f} \\
&\leq C f.
\end{split}
\end{align}
where in the final inequality we used the fact that $f = \epsilon f_0$ with $f_0 > 0$ depending only on the boundary data for $u$.

The second line in (\ref{LP4}) can be estimated as
\begin{align} \label{in2}
- 2 K |\N u_t|^2 |\eta|^2 + 2 K \langle \N u_t , \eta \rangle^2 + \N^2 f(\eta,\eta) \leq |\eta|^2  |\N^2 f| \leq C f.
\end{align}
where again we used the special form of $f$.
Combining (\ref{LP4}), (\ref{in1}), and (\ref{in2}) we conclude
\begin{align} \label{LP5}
\mathcal{L}_u \phi \leq C f + (-u_{tt})\big[  C + C |\N^2 u| \big].
\end{align}

Next, we recall the identity (\ref{LDu}):
\begin{align} \label{LD2} \begin{split}
\mathcal{L}_u |\N u|^2 =&\ 2 |\N u_{t}|^2  ( -\Delta u + K) + (-u_{tt}) \left(2 |\nabla^2 u|^2 + 2 K|\N u|^2 - 2 \langle \N K, \N u \rangle \right)  \\
&\quad + 4 \N_k \N_i u \N_k u_t \N_i u_t - 2 \langle \N u , \N f \rangle.
\end{split}
\end{align}
Writing the equation as in (\ref{equt}), we can estimate the first term on the second line above by
\begin{align} \label{ct1} \begin{split}
4 \N_k \N_i u \N_k u_t \N_i u_t &\geq - 4 |\N^2 u| |\N u_t|^2 \\
&= - 4 |\N^2 u| \Big\{ (-u_{tt})\big( - \Delta u + K \big) - f \Big\} \\
&\geq - 4 (-u_{tt}) |\N^2 u| ( - \Delta u + K ).
\end{split}
\end{align}
Since $|\Delta u|$ is bounded, it follows that
\begin{align} \label{SL1}
4 \N_k \N_i u \N_k u_t \N_i u_t \geq - C (-u_{tt})|\N^2 u|.
\end{align}
Using the bound for the gradient, the last term in the second line of (\ref{LD2}) can easily be estimated as
\begin{align} \label{SL2}
- 2 \langle \N u , \N f \rangle \geq -  C |\N f| \geq - C f.
\end{align}
Substituting (\ref{SL1}) and (\ref{SL2}) into (\ref{LD2}), we get
\begin{align} \label{LD3}
\mathcal{L}_u |\N u|^2 \geq (-u_{tt}) \left( 2 |\nabla^2 u|^2 - C |\N^2 u| \right) - C f.
\end{align}
Using Lemma \ref{dereqlemzed}, we therefore have
\begin{align} \label{BC1}
\mathcal{L}_u \big( -  |\N u|^2 +  2 t (1-t) \big) &\leq   (-u_{tt}) \left( - 2  |\nabla^2 u|^2 +  C |\N^2 u|  \right)  - 4 (\Delta u + K ) + C  f.
\end{align}
Combining this with (\ref{LP5}) gives
\begin{align} \label{LP6}
\mathcal{L}_u \Psi \leq (-u_{tt}) \big\{ - 2 |\N^2 u|^2 + C  |\N^2 u| + C \big\} - 4 (\Delta u + K ) + C f.
\end{align}
If $|\N^2 u|$ is large enough, we may assume that
\begin{align*}
- 2|\N^2 u|^2 + C |\N^2 u| + C  \leq -  |\N^2 u|^2.
\end{align*}
Therefore,
\begin{align} \label{LP7}
\mathcal{L}_u \Psi \leq  - |\N^2 u|^2 (-u_{tt}) - 4 (\Delta u + K ) + C  |\N^2 u| + C  f.
\end{align}
By $(\star_{\epsilon})$,
\begin{align*}
-u_{tt} = \frac{|\N u_t|^2 + f}{(-\Delta u + K)} \geq \frac{f}{(-\Delta u + K)},
\end{align*}
hence
\begin{align} \label{LP8}
\mathcal{L}_u \Psi \leq - \frac{  f}{(-\Delta u + K)}|\N^2 u|^2  - 4 (\Delta u + K ) + C  f.
\end{align}
Since
\begin{align*}
4 \sqrt{f} |\N^2 u| \leq \frac{  f}{(-\Delta u + K)}|\N^2 u|^2  + 4 (\Delta u + K ),
\end{align*}
it follows that
\begin{align} \label{LP9}
\mathcal{L}_u \Psi \leq - 4 \sqrt{f} |\N^2 u| + C f.
\end{align}
Since $f = \epsilon f_0 \approx \epsilon$, it follows that $\sqrt{f} \geq  c_0 f$ for some $c_0 > 0$.  Therefore,
\begin{align} \label{LP10}
\mathcal{L}_u \Psi \leq ( - 4 c_0  |\N^2 u| + C ) f,
\end{align}
and once $|\N^2 u|$ is large enough, we conclude $\mathcal{L}_u \Psi < 0$.  As a consequence, either $|\N^2 u|$ is bounded at an interior minimum of $\Psi$, or else $\Psi$ attains its minimum when $t=0$ or $t=1$.  In either case we see that $\Psi$ is bounded from below, hence $|\N^2 u|$ is bounded.
\end{proof}

\subsection{Existence}

Using the {\em a priori} estimates, we can use the continuity argument to prove the following existence result:

\begin{prop} \label{posReg}  For each $\epsilon > 0$ sufficiently small, there is a unique $C^{\infty}$-solution $u = u_{\epsilon}(x,t)$ of $(\star_{\epsilon})$.  Furthermore,
there is a constant $C$ depending on $u_0, u_1$ (but independent of $\epsilon$) such that
\begin{align} \label{C11}
\max_{ M \times [0,1]} \big\{ |u| + |\N u| + |u_t| + |u_{tt}| + |\N u_t| + |\N^2 u| \big\} \leq C.
\end{align}
\end{prop}

\begin{proof}  Fix $0 < \epsilon_0 < 1$, and let
\begin{align*}
\mathcal{I} = \big\{ \epsilon \in [\epsilon_0, 1] \ :\ \exists u \in C^{4,\alpha} \cap \Gpos,\ u \mbox{ solves $(\star_{\epsilon})$} \big\}.
\end{align*}
By Lemma \ref{initLemma}, there is a (smooth) admissible solution $u$ of $(\star_{\epsilon})$ with $\epsilon =1$.  Therefore, $\mathcal{I}$ is non-empty.

We claim that $\mathcal{I}$ is closed: let $\{ u_i = u_{\epsilon_i} \}$ be a sequence of admissible solutions with $\epsilon_i \geq \epsilon_0 > 0$.  The preceding {\em a priori }
estimates imply there is a constant $C$ (independent of $\epsilon$) such that
\begin{align*}
\max \big\{ |u_i| + |\N u_i| + |(u_i)_t| + |(u_i)_{tt}| + |(\N u_i)_t| + |\N^2 u_i| \big\} \leq C.
\end{align*}
To apply Evans/Krylov and obtain H\"older estimates for the second derivatives, we need to verify the concavity of the operator.  To do so we will rewrite the operator as
\begin{align} \label{Opdef}
\mathcal{M}[u] = u_{tt} + \frac{ |\N u_t|^2 + f}{-\Delta u + K}.
\end{align}
We view $\mathcal{M}$ as a function of the space-time Hessian of $u$, and to simplify the calculations let us fix a point $x_0 \in M$ and introduce normal coordinates $\{ x^1, x^2 \}$ based at $x_0$.  We also denote derivatives with respect to $t$ with the subscript $0$; then at $x_0$
\begin{align} \label{Mr}
\mathcal{M}[u] = - u_{00} - \frac{u_{01}^2 + u_{02}^2 + f(x_0)}{-u_{11} - u_{22} + K(x_0)}.
\end{align}
To simplify just write $f(x_0) = f$ and $K(x_0) = K$.  Expressed in this way it is easy to see that the equation remains elliptic: differentiating with respect to the entries of the space-time Hessian $\{ r_{ab} \}$, $0 \leq a,b \leq 2$, we have
\begin{align*}
\mathcal{M}^{00} &= 1 \\
\mathcal{M}^{0i} &= \frac{2 r_{0i}}{-r_{11} - r_{22} + K}, \quad 1 \leq i \leq 2 \\
\mathcal{M}^{ij} &= \frac{ r_{01}^2 + r_{02}^2 + f}{-r_{11} - r_{22} + K} \delta_{ij}, \quad 1 \leq i,j \leq 2.
\end{align*}
Differentiating again,
\begin{align*}
\mathcal{M}^{00,00} &= 0 \\
\mathcal{M}^{00,0i} &= 0, \quad 1 \leq i \leq 2 \\
\mathcal{M}^{00,ij} &= 0, \quad 1 \leq i,j \leq 2  \\
\mathcal{M}^{0i,k\ell} &= \frac{2 r_{0i}}{(-r_{11} - r_{22} + K)^2}\delta_{k \ell}, \quad 1 \leq i,k,\ell \leq 2  \\
\mathcal{M}^{0i,0j} &= \frac{2 \delta_{ij} }{-r_{11} - r_{22} + K}, \quad 1 \leq i,j \leq 2  \\
\mathcal{M}^{ij, k\ell} &= \frac{2(r_{01}^2 + r_{02}^2 + f)}{(- r_{11} - r_{22} + K)^3}\delta_{ij} \delta_{k\ell}, \quad 1 \leq i,j,k,\ell \leq 2.
\end{align*}
Then
\begin{align} \label{Mee1} \begin{split}
\mathcal{M}^{ab, cd} \eta^{ab} \eta^{cd} =&\ \frac{4}{(-r_{11} - r_{22} + K)^2} \big[ r_{01} \eta^{01} + r_{02} \eta^{02} \big] \big[ \eta^{11} + \eta^{22} \big] \\
 &\  + \frac{2}{-r_{11} - r_{22} + K}\big[ (\eta^{01})^2 + (\eta^{02})^2 \big] + \frac{ 2( r_{01}^2 + r_{02}^2 + f)}{(-r_{11} - r_{22} + K)^3} \big(\eta^{11} + \eta^{22} \big)^2.
 \end{split}
\end{align}
The first term above can be estimated by
\begin{align*}
\frac{4}{(-r_{11} - r_{22} + K)^2}  &\big[ r_{01} \eta^{01}  + r_{02} \eta^{02} \big] \big[ \eta^{11} + \eta^{22} \big] \geq   \\
& - \frac{ 2( r_{01}^2 + r_{02}^2 + f)}{(-r_{11} - r_{22} + K)^3} \big(\eta^{11} + \eta^{22} \big)^2 - \frac{2}{(-r_{11} - r_{22} + K)} \frac{\big[ r_{01} \eta^{01} + r_{02} \eta^{02} \big]^2}{( r_{01}^2 + r_{02}^2 + f)}.
\end{align*}
Substituting into (\ref{Mee1}) gives
\begin{align} \label{Mee2}  \begin{split}
\mathcal{M}^{ab, cd} \eta^{ab} \eta^{cd} &\geq \frac{2}{(-r_{11} - r_{22} + K)} \Bigg\{ \big[ (\eta^{01})^2 + (\eta^{02})^2 \big] -  \frac{\big[ r_{01} \eta^{01} + r_{02} \eta^{02} \big]^2}{( r_{01}^2 + r_{02}^2 + f)} \Bigg\} \\
&= 2 \frac{ \big[ (\eta^{01})^2 + (\eta^{02})^2 \big]( r_{01}^2 + r_{02}^2 + f) -  \big[ r_{01} \eta^{01} + r_{02} \eta^{02} \big]^2 }{(-r_{11} - r_{22} + K)}   \\
&\geq 2 \frac{ \Big[ r_{01}^2 (\eta^{02})^2 + r_{02}^2 (\eta^{01})^2  - 2 r_{01} r_{02} \eta^{01} \eta^{02} \Big]  }{(-r_{11} - r_{22} + K)} \\
&\geq 0.
\end{split}
\end{align}
It follows that $\mathcal{M}$ is a convex function of the space-time Hessian, and applying Evans-Krylov \cite{Evans} \cite{Krylov} we conclude there is a constant $C = C(\epsilon)$ such that
\begin{align*}
\| u_i \|_{C^{2,\alpha}} \leq C.
\end{align*}
Applying the Schauder estimates we obtain bounds on derivatives of all orders, and it follows that the set $\mathcal{I}$ is closed.

To verify that $\mathcal{I}$ is open, it suffices to study the linearized equation; i.e., given $\psi \in C^{\infty}(M \times [0,1])$, we need to solve
\begin{align*}
\mathcal{L}_{u_i} \varphi = \psi
\end{align*}
with $\varphi$ satisfying Dirichlet boundary conditions.  The solvability of this linear problem follows from \cite{GT}, Theorem 6.13. 

Since $\mathcal{I}$ is open, closed, and non-empty, it follows that $\mathcal{I} = [\epsilon_0,1]$.  This proves the existence of solutions.  Uniqueness will follow from the following comparison lemma: 

\begin{lemma} \label{compare} Suppose $u, \tilde{u} \in C^{\infty}$ are admissible and satisfy
\begin{align} \label{Guv} \begin{split}
\mathcal{G}_{f_1}(u) &= 0, \\
\mathcal{G}_{f_2}(\tilde{u}) &= 0, 
\end{split}
\end{align}
where $f_1 \leq f_2$.  Assume further that on the boundary,  
\begin{align} \label{uvbdy} \begin{split}
u(x,0) &= \tilde{u}(x,0), \\
u(x,1) &= \tilde{u}(x,1). 
\end{split}
\end{align}
Then on $M \times [0,1]$, 
\begin{align} \label{ulv}
u(x,t) \geq \tilde{u}(x,t).
\end{align}
\end{lemma}

\begin{proof} 
For $0 \leq \theta \leq 1$,
define
\begin{align*}
w_{\theta} = (1 - \theta) u + \theta \tilde{u}.
\end{align*}
Then $w_0 = u$ and $w_1 = \tilde{u}$.  By (\ref{Guv}), 
\begin{align*}
\mathcal{G}_0(w_1) - \mathcal{G}_0(w_0) = \mathcal{G}_0(\tilde{u}) - \mathcal{G}_0(u) = f_2 - f_1 \geq 0.
\end{align*}
Therefore, 
\begin{align*}
0 \leq \mathcal{G}_0[\tilde{u}] - \mathcal{G}_0[u] &= \int_0^1 \frac{d}{d\theta} \mathcal{G}_0[w_{\theta}]\ d\theta \\
&= \int_0^1 \mathcal{L}_{w_{\theta}}( \tilde{u} - u ) d\theta \\
&= \int_0^1 \Big\{ (1-\theta) \mathcal{L}_u (\tilde{u}-u) + \theta \mathcal{L}_{\tilde{u}} (\tilde{u}-u) \Big\} d\theta \\
&= \frac{1}{2} \big\{ \mathcal{L}_u + \mathcal{L}_{\tilde{u}} \big\}(\tilde{u}-u).
\end{align*}
Since $\tilde{u} - u = 0$ on the boundary, by the maximum principle we conclude that $\tilde{u} - u \leq 0$ on $M \times [0,1]$.  \end{proof}

If $u, \tilde{u}$ are solutions of $(\star_{\epsilon})$ with the same boundary conditions, then we may apply the preceding comparison lemma with $f_1 = f_2 = \epsilon f_0$ and 
conclude that $u = v$.  This completes the proof of Proposition \ref{posReg}. 
\end{proof}

Since we have existence and uniqueness of classical solutions for the regularized problem, we make the following definition:


\begin{defn} \label{RG} We say that a $C^{1,1}$ solution $u(x,t)$ to (\ref{Gf}) with $f \equiv 0$ is
\emph{regularizable} if there exists $f_0 \in C^{\infty}(M \times [0,1])$ with $f_0 > 0$ and a smooth
function $u(x,t,\ge) : M \times [0,1] \times
[0,\ge_0) \to \mathbb R$ with the following properties:  \\

\noindent $(i)$ For each $\ge \in [0,\ge_0)$ $u_{\epsilon} = u(\cdot,\cdot,\epsilon)$ satisfies
\begin{align*}
u_{\epsilon}(x,0) = u(x,0), \qquad u_{\epsilon}(x,1) = u(x,1), \qquad \mathcal{G}_{\ge f_0}(u_{\epsilon}) = 0.
\end{align*}

\noindent $(ii)$ There is a constant $C_2 > 0$, independent of $\epsilon$, such that
\begin{align*}
|u_{\epsilon}| + |\N u_{\epsilon}| + |(u_{\epsilon})_t| + |\N^2 u_{\epsilon}| + |\N (u_{\epsilon})_t| + |(u_{\epsilon})_{tt}| \leq C_2.
\end{align*}

\noindent $(iii)$  For each $0 < \alpha < 1$, $u_{\epsilon} \rightarrow u$ in $C^{1,\alpha}$.
\end{defn}

Given $u_0,u_1 \in \gG_1^+$, by Proposition \ref{posReg} we know there is a regularizable geodesic $u \in C^{1,1}(M \times [0,1])$ connecting $u_0$ and $u_1$.  We claim that $u$ is unique.  To
see this, suppose $u, \tilde{u}$ are regularizable geodesics with the same boundary conditions.  By definition, there are regularizations $u(x,t,\ge)$,
$\tilde{u}(x,t,\delta)$, satisfying 
\begin{align*}
\mathcal{G}_{\epsilon f_0}(u_{\epsilon}) &= 0, \\
\mathcal{G}_{\delta \tilde{f}_0}(\tilde{u}_{\delta}) &= 0, 
\end{align*}
where $f_0 , \tilde{f}_0 > 0$.  For fixed $\delta > 0$ and all $\epsilon > 0$ sufficiently small, $\epsilon f_0 < \delta \tilde{f}_0$.  By Lemma \ref{compare}, it follows that
\begin{align*}
u(x,t,\ge) \geq \tilde{u}(x,t,\delta) 
\end{align*}
for fixed $\delta > 0$ and all $\epsilon > 0$ small.  Letting $\epsilon \rightarrow 0$, we get
\begin{align*}
u(x,t) \geq \tilde{u}(x,t,\delta).
\end{align*}
Letting $\delta \rightarrow 0$ gives 
\begin{align*}
u(x,t) \geq \tilde{u}(x,t).
\end{align*}
Reversing the roles of $\delta$ and $\epsilon$, we conclude that $u = \tilde{u}$.  Summarizing:

\begin{prop} \label{posExist} Given $u_0, u_1 \in \Gpos$, there is a unique regularizable $C^{1,1}$-geodesic with
\begin{align} \label{rgbdy} \begin{split}
u(x,0) &= u_0(x), \\
u(x,1) &= u_1(x).
\end{split}
\end{align}
\end{prop}

\subsection{Existence of geodesics in the negative cone}

The regularity estimates for the geodesic equation in the negative cone follow almost immediately from the estimates for the positive cone equation.  Recall from (\ref{geodesicNegpde}) that the geodesic equation is
\begin{align} \label{NCG1}
w_{tt} ( - \Delta w + K ) + |\N w_t|^2 = 0.
\end{align}
Let $u = -w$ and $\tilde{K} = - K > 0$; then we can rewrite this as
\begin{align} \label{NCG2}
u_{tt} ( - \Delta u + \tilde{K} ) + |\N u_t|^2 = 0,
\end{align}
which is precisely the geodesic equation for the positive cone.  Therefore,

\begin{prop}  \label{negExist}  Given $w_0, w_1 \in \Gneg$, there is a unique regularizable $C^{1,1}$-geodesic $w = w(x,t)$ with
\begin{align*}
w(x,0) &= w_0(x), \\
w(x,1) &= w_1(x).
\end{align*}
\end{prop}

\section{Metric space structure} \label{metricspace}

In this section we use the existence of regularizable geodesics to define a metric space structure on $\gG_1^+$.  First we show that our definition does indeed define a metric space, and then we establish nonpositivity of curvature in the sense of Alexandrov.  For simplicity we focus entirely on the positive cone setting, the proofs for the negative cone being directly analogous.

\subsection{Metric space structure}

\begin{defn} Given $u_0,u_1 \in \gG_1^+$, let $d(u_0,u_1)$ denote the distance of
the unique regularizable geodesic connecting $u_0$ to $u_1$. 
\end{defn}

We will establish that $d$ does indeed define a metric space structure in what follows.  First we establish nondegeneracy of the distance, for which we require a preliminary lemma.  Then we establish the triangle inequality, finishing the proof.

\begin{lemma} \label{constantenergy} Let $u$ be a regularizable geodesic.  Then
$u$ has constant energy density.
\begin{proof} Fixing a regularization we directly compute using the uniform $C^{1,1}$ bounds to yield
\begin{align*}
\brs{ \frac{d}{dt} \brs{\brs{ \frac{\del u}{\del t}}}^2} =&\ 2 \brs{ \int_M u_t \left(u_{tt}K_u + \brs{\N u_t}^2
\right) dV_u}\\
=&\ 2 \brs{ \int_M u_t \ge f_0 dV_u}\\
\leq&\ 2 \ge C.
\end{align*}
Taking $\ge$ to zero and using the convergence properties finishes the lemma.
\end{proof}
\end{lemma}

\begin{prop} \label{distancenondegeneracy} Given
$u_0, u_1 \in \gG_1^+$ and $u$ the unique
regularizable geodesic connecting $u_0$ to $u_1$, one has
\begin{align*}
d(u_0,u_1) \geq&\ (2 \pi \chi)^{-\frac{1}{2}} \max \left\{ \int_{u_0 > u_1} (u_0 - u_1)
K_{u_1} dV_{u_1}, \int_{u_1 > u_0} (u_1 - u_0) K_{u_0}
dV_{u_0} \right\}.
\end{align*}
\begin{proof} Observe that the geodesic equation implies $u_{tt} \leq 0$, and
so we obtain the pointwise inequality
\begin{align*}
{u_t}(1) \leq u_1 - u_0 \leq {u_t}(0).
\end{align*}
Thus using H\"older's inequality and the Gauss-Bonnet theorem we have
\begin{align*}
\brs{\brs{\frac{\del u}{\del t}}}(1) =&\ \left( \int_M {u_t}(1)^2 K_{u_1} dV_{u_1} \right)^{\frac{1}{2}}\\
\geq&\ (2 \pi \chi)^{-\frac{1}{2}} \int_M \brs{u_t(1)} K_{u_1} dV_{u_1}\\
\geq&\ (2 \pi \chi)^{-\frac{1}{2}} \int_{u_0 > u_1} (u_0 - u_1) K_{u_1} dV_{u_1}.
\end{align*}
A similar argument yields
\begin{align*}
\brs{\brs{\frac{\del u}{\del t}}}(0) \geq (2 \pi \chi)^{-\frac{1}{2}} \int_{u_1 > u_0} (u_1 - u_0) K_{u_0}
dV_{u_0}.
\end{align*}
Since geodesics automatically have constant energy density by Lemma \ref{constantenergy},
the result follows.
\end{proof}
\end{prop}

Our next goal is to establish the triangle inequality for the metric $d$.  To do
this we require the existence of approximate geodesics connecting paths
in $\gG_1^+$.  The proof of the following proposition is a straightforward
adaptation of the proof of Proposition \ref{posReg}

\begin{prop} \label{geodesicfamilies} Let $\gg_i : [0,1] \to \gG_1^+$ denote two
smooth curves of
admissible conformal factors.  There exists a smooth function $f : M \times [0,1] \to \mathbb R \geq 0$, $\ge_0 > 0$, and a two-parameter family of curves
$\gU : [0,1] \times [0,1] \times (0,\ge_0] \to \gG_1^+$ such that for each $s$, the family $\gU(\cdot, s,\cdot)$ is a regularization of the unique regularizable geodesic connecting $\gg_1(s)$ to $\gg_2(s)$.
\end{prop}

\begin{prop}[Triangle Inequality] \label{triangle} Let $u : [0,1] \to \gG_1^+$ be a smooth curve.  Given $v \in \gG_1^+$, one
has
\begin{align*}
 d(v, u_1) \leq d(v,u_0) + L(u).
\end{align*}
In particular, one has
\begin{align*}
 d(v, u_1) \leq d(v,u_0) + d(u_0,u_1).
\end{align*}

\begin{proof} It suffices to consder the case when $v = 0$ by changing
basepoints.  We apply Proposition \ref{geodesicfamilies} in the case that
$\gg_1(s) = 0$, $\gg_2(s) = u$, obtaining the two-paramter family $\gU$.  Let $L(s,\ge)$ denote
the length of $\gU(\cdot,s,\ge)$, and let $\gl(s)$ denote the length of the
given curve $u$ up to time $s$, i.e.
\begin{align*}
\gl(s) := \int_0^s \left[ \int_M \brs{\frac{\del u}{\del \gs}}^2 K_{u} dV_u
\right]^{\frac{1}{2}} d \gs.
\end{align*}
The main step is to obtain a lower bound for $L(s,\ge) + \gl(s)$.  To do this we use
the $\ge$-approximations and compute
\begin{align*}
\frac{d L(s,\ge)}{d s} =&\ \int_0^1\frac{1}{2} E_u(t,s,\ge)^{-\frac{1}{2}} \int_M
\left[ 2 \frac{\del \gU}{\del t} \frac{\del^2 \gU}{\del t \del s} K_{\gU} -
\left( \frac{\del \gU}{\del t} \right)^2 \gD \frac{\del \gU}{\del s} \right]
dV_{\gU} dt\\
=&\ \int_0^1 E_u(t,s,\ge)^{-\frac{1}{2}} \left\{ \frac{d}{dt} \int_M \frac{\del
\gU}{\del t} \frac{\del \gU}{\del s}  K_{\gU} dV_{\gU} - \int_M \frac{\del
\gU}{\del s} \left[ \frac{\del^2 \gU}{\del t^2} K_{\gU} + \frac{1}{2} \brs{\N
\frac{\del \gU}{\del t}}^2 \right] dV_{\gU} \right\} dt\\
=&\ \left. E_u(t,s,\ge)^{-\frac{1}{2}} \int_M \frac{\del \gU}{\del t} \frac{\del
\gU}{\del s} K_{\gU} dV_{\gU} \right|_{t=0}^{t=1} - \int_0^1 \int_M \frac{\del \gU}{\del
s} \left[ \frac{\del^2 \gU}{\del t^2} K_{\gU} + \frac{1}{2} \brs{\N \frac{\del
\gU}{\del t}}^2 \right] dV_{\gU}\\
&\ + \int_0^1 \left\{ E_u(t,s,\ge)^{-\frac{3}{2}} \left(\int_M \frac{\del \gU}{\del
t} \frac{\del \gU}{\del s} K_{\gU} dV_{\gU} \right) \left( \int_M \frac{\del
\gU}{\del t} \left( \frac{\del^2 \gU}{\del t^2} K_{\gU} + \frac{1}{2} \brs{\N \frac{\del
\gU}{\del t}}^2 \right) dV_{\gU} \right) \right\} dt\\
\geq&\ E_u(1,s,\ge)^{-\frac{1}{2}} \int_M \frac{\del \gU(1,s,\ge)}{\del t}
\frac{\del \gU}{\del s} K_{\gU} dV_{\gU} - C \ge\\
\geq&\ - \left[ \int_M \brs{ \frac{\del \gU}{\del s}}^2 K_{\gU} dV_{\gU}
\right]^{\frac{1}{2}} - C \ge.
\end{align*}
Since
\begin{align*}
\frac{d \gl}{ds} = \left[ \int_M \brs{ \frac{\del \gU}{\del s}}^2 K_{\gU}
dV_{\gU} \right]^{\frac{1}{2}},
\end{align*}
we conclude that
\begin{align*}
\frac{d}{ds} \left[ L(s,\ge) + \gl(s) \right] \geq - C \ge.
\end{align*}
Using the convergence properties and sending $\ge$ to zero yields the result.
\end{proof}
\end{prop}

\begin{prop} \label{metricspacecor} $(\gG_1^+, d)$ is a metric space.
\begin{proof} This follows immediately from Propositions \ref{distancenondegeneracy} and \ref{triangle}.
\end{proof}
\end{prop}

\subsection{Nonpositive curvature}

Next we establish the nonpositive curvature of the metric spaces $(\gG_1^{\pm}, d)$ in
the sense of Alexandrov.  First let's record some notation and recall the
definition of nonpositive curvature.  Given $A,B \in \gG_1^+$, let $AB(s)$ denote
$\gg(s)$, where $\gg : [0,1] \to \gG_1^+$ is the unique regularizable $C^{1,1}$ geodesic connecting
$A$ to $B$.  Note that a priori $AB(s)$ is only $C^{1,1}$, but this suffices
for our purposes in establishing metric space properties.  The condition of
nonpositive curvature then amounts to the claim that for all $A,B,C$ one has the
inequality:
\begin{align*}
d^2(A,BC(s)) \leq (1-s) d(A,B)^2 + s d(A,C)^2 - s(1-s) d(B,C)^2.
\end{align*}
\noindent The first step in establishing this is to show that Jacobi fields are
convex along geodesics.

\begin{lemma} \label{jacobilemma} Given $\gg_i : [0,1] \to \gG_1^+$, let $\gU: [0,1] \times [0,1] \times (0,\ge_0] \to
\gG_1^+$ be the family guaranteed by Proposition \ref{geodesicfamilies}.
Then for $\ge \in (0,\ge_0]$ one has
\begin{align*}
\IP{\IP{ \frac{\del \gU}{\del s}, \N_{\frac{\del \gU}{\del t}} \frac{\del \gU}{\del
s}}} \geq \brs{\brs{\frac{\del \gU}{\del s}}}^2 - C \ge.
\end{align*}
\begin{proof} To simplify notation we set $X = \frac{\del \gU}{\del t}, Y =
\frac{\del \gU}{\del s}$.  Then we compute
\begin{align*}
\frac{1}{2} \frac{\del^2}{\del t^2} \dbrs{Y}^2 =&\ \dIP{\N_X Y, \N_X Y} + \dIP{\N_X
\N_X Y, Y}\\
=&\ \dbrs{\N_X Y}^2 - K(X,Y) + \dIP{\N_Y \N_X X, Y}\\
\geq&\ \dbrs{\N_X Y}^2 + \frac{d}{ds} \dIP{\N_X X, Y} - \dIP{\N_X X, \N_Y Y}\\
\geq&\ \dbrs{\N_X Y}^2 + \frac{d}{ds} \int_M \frac{\del \gU}{\del s} \ge f
dV_{\gU} - C \ge\\
\geq&\ \dbrs{\N_X Y}^2 - C \ge.
\end{align*}
Using this it follows that
\begin{align*}
\frac{\del^2}{\del t^2} \dbrs{Y} \geq - C \ge.
\end{align*}
Combining this with the fact that $Y(0) = 0$, we have
\begin{align*}
\frac{\del}{\del t} \dbrs{Y(t)}_{|t=1} \geq \dbrs{Y(1)} - C \ge.
\end{align*}
Rearranging this yields
\begin{align*}
\dIP{\N_X Y, Y} \geq \dbrs{Y}^2,
\end{align*}
as required.
\end{proof}
\end{lemma}

\begin{prop} \label{NPCprop} The metric spaces $(\gG_1^{\pm}, d)$ are nonpositively curved
in the sense of Alexandrov.
\begin{proof} Fix three points $A,B,C \in \gG_1^+$, let $\gg_1(s) = A$ for all $s$,
and let $\gg_2(s)$ be the an $\ge$-approximate geodesic $B$ to $C$.  Let $\gU$
denote the family associated to these two paths guaranteed by Proposition
\ref{geodesicfamilies}.  Treating $\ge$ as small but fixed, let $E(s)$ denote
the total energy of the path $\gU(x,t,s,\ge)$ connecting $A$ to $\gg_2(s)$.  We aim to show that
$E(s)$ is convex up to an error of order $\ge$.  We simplify notation and set $X
= \frac{\del \gU}{\del t}, Y = \frac{\del \gU}{\del s}$.  Then we compute
\begin{align*}
\frac{1}{2} \frac{d E(s)}{ds} =&\ \int_0^1 \dIP{\N_{Y} X, X}\\
=&\ \int_0^1 X \dIP{X,Y} - \dIP{\N_X X, Y}\\
=&\ \dIP{X,Y}_{|t=1} - \int_0^1 \int_M \frac{\del \gU}{\del s} \ge f dV_{\gU}.
\end{align*}
Next we compute, using Lemma \ref{jacobilemma},
\begin{align*}
\frac{1}{2} \frac{d^2 E(s)}{ds^2} =&\ \frac{d}{ds} \dIP{X,Y}_{|t=1} - \ge
\int_0^1 \int_M \left(\frac{\del^2 \gU}{\del s^2} + 2 \left(\frac{\del \gU}{\del
s} \right)^2 \right) f dV_{\gU}\\
\geq&\ \dIP{\N_Y X, Y}_{|t=1} + \dIP{X,\N_Y Y}_{|t=1} - C \ge\\
\geq&\ \dbrs{Y}^2_{|t=1} - C \ge\\
\geq&\ d(B,C)^2 - C \ge.
\end{align*}
It follows directly that
\begin{align*}
E(s) \leq&\ (1-s) E(0) + s E(1) - s(1-s) \left( d(B,C)^2 - C \ge \right).
\end{align*}
Note that for each $s$ the energies $E(s)$ converge as $\ge \to 0$ to $d(A, BC(s))^2$, thus we conclude by
sending $\ge$ to zero in the above inequality that
\begin{align*}
d^2(A,BC(s)) \leq (1-s) d(A,B)^2 + s d(A,C)^2 - s(1-s) d(B,C)^2,
\end{align*}
as required.
\end{proof}
\end{prop}

\begin{proof}[Proof of Theorem \ref{msthm}] This follows directly from Propositions \ref{posExist}, \ref{negExist}, \ref{metricspacecor}, and \ref{NPCprop}.
\end{proof}

\section{Functional determinant and the inverse Gauss curvature flow}

In \cite{Polyakov}, Polyakov proved a remarkable formula for the ratio of regularized determinants of two conformal metrics.  Given conformal metrics $g_0$ and $g_u = e^{2u}g$ on a surface $M$,
\begin{align} \label{Polyform}
\log \frac{\det (-\Delta_{g_u})}{\det (-\Delta_0)} =
-\frac{1}{12\pi} \int_{M} (|\nabla_0 u|^2 + 2K_0 u)\ dA_0.
\end{align}
The integral in this formula is often referred to as the {\em Liouville energy}, and we will denote it by $J$:
\begin{align} \label{Jdef}
J[u] = \int_M |\nabla_0 u|^2 dA_0 + 2 \int K_0 u\ dA_0.
\end{align}
Since $J$ is not scale-invariant, it is convenient to also consider the normalized version of $J$, which we denote by $S$:
and we denote the (normalized) version by $F : W^{1,2} \rightarrow \mathbb{R}$
\begin{align} \label{Fdef}
F[u] = \int_M |\nabla_0 u|^2 dA_0 + 2 \int_M K_0 u dA_0 - \big( \int_M K_0 dA_0 \big) \log \Big( \fint_M e^{2u} dA_0  \Big),
\end{align}
where $\fint$ denotes the normalized integral
\begin{align*}
\fint_M f dA_0 = \frac{\displaystyle \int_M f dA_0 }{ \displaystyle \int_M dA_0 }.
\end{align*}
Note that $S[u + c] = S[u]$ for any real number $c$.

A first variation calculation of $J$ at $u$ is
\begin{align} \label{Jdot} \begin{split}
(J')_u(u') &= \frac{d}{ds} J[ u + s u'] \Big|_{s=0} \\
&= 2 \int u' \big( - \Delta_0 + K_0 \big) dA_0 \\
&= 2 \int_M u' K_u dA_u,
\end{split}
\end{align}
and a first variation of $F$ is
\begin{align} \label{Sdot}
(F')_u(u') = 2 \int_M u' (K_u - \bar{K}_u) dA_u,
\end{align}
where
\begin{align*}
\bar{K}_u = \frac{\displaystyle \int_M K_u dA_u}{\displaystyle \int_M dA_u}.
\end{align*}

%
%

We want to consider the (negative) gradient flow for $F$ with respect to the metric we have defined for $\Gpos$.  In view of (\ref{Sdot}) and (\ref{metdef}),
\begin{align} \label{gradL2pos} \begin{split}
(F')_u(u') &= 2 \int_M u' (K_u - \bar{K}_u) dA_u \\
&= \int_M u' \frac{ (K_u - \bar{K}_u)}{K_u} K_u dA_u \\
&= \big\langle u' , \frac{ (K_u - \bar{K}_u)}{K_u} \big\rangle_{\Gpos}.
\end{split}
\end{align}
It follows that the (negative) gradient flow for $F$ is given by
\begin{align} \label{posgradflow}
\frac{\partial u}{\partial t} = \frac{ (\bar{K}_u)- K_u}{K_u},
\end{align}
or written in metric terms,
\begin{align} \label{dgdtpos}
\frac{\partial}{\partial t}g = 2 \big( \frac{ \bar{K} - K}{K}\big) g.
\end{align}
We will refer to (\ref{posgradflow}) and (\ref{dgdtpos}) as the {\em inverse Gauss curvature flow}, or IGCF.  For the cone of metrics with negative curvature, the
flow is defined by
\begin{align} \label{neggradflow}
\frac{\partial u}{\partial t} = \frac{ (K_u - \bar{K}_u)}{K_u},
\end{align}
or
\begin{align} \label{dgdtneg}
\frac{\partial}{\partial t}g = 2 \big( \frac{ K - \bar{K}}{K}\big) g.
\end{align}

\subsection{Formal Properties}

In this section we establish the geodesic convexity of the Liouville energy, the convexity of the normalized Liouville energy along flow lines as well as the monotonicity of distances along flow lines.  Remarkably, all three properties rely on a sharp application of a curvature-weighted Poincare inequality \cite{Andrews}.  We include the short proof as this result seems to not be well-known.

\begin{prop} \label{andrewsineq} (Andrews \cite{Andrews}, cf. \cite{CLN} pg.
517) Let $(M^n, g)$ be a closed
Riemannian manifold with
positive Ricci curvature.  Given $\phi \in C^{\infty}(M)$ such that $\int_M \phi
dV = 0$, then
\begin{align*}
\frac{n}{n-1} \int_M \phi^2 dV \leq&\ \int_M \left( \Rc^{-1} \right)^{ij} \N_i
\phi
\N_j \phi dV,
\end{align*}
with equality if and only if $\phi \equiv 0$ or $(M^n, g)$ is isometric to the
round sphere.
\begin{proof} Since $\int_M \phi = 0$ there exists $\psi$ such that $\gD \psi =
\phi$.  Observe that
\begin{align*}
\int_M \brs{\N^2 \psi - \frac{\gD \psi}{n} g}^2 =&\ \int_M \N_i \N_j \psi \N_i
\N_j \psi - \frac{\phi^2}{n}\\
=&\ \frac{n-1}{n} \int_M \phi^2 dV - \int_M \Rc_{ij} \N_i \psi \N_j \psi.
\end{align*}
Moreover
\begin{align*}
\int_M & \left(\Rc^{-1} \right)^{ij} \left[ \N_i \phi + b \Rc_{ik} \N_k \psi
\right] \left[ \N_j \phi + b \Rc_{jl} \N_l \psi \right] \\
=&\ \int_M\left( \Rc^{-1} \right)^{ij} \N_i \phi \N_j \phi - 2 b \phi^2 + b^2
\Rc_{kl} \N_k \psi \N_l \psi.
\end{align*}
Combining these and choosing $a=b^{-1} = \frac{n-1}{n}$ yields
\begin{align*}
0 \leq&\ \int_M \brs{\N^2 \psi - \frac{\gD \psi}{n} g}^2 + a^2 \int_M
\left(\Rc^{-1} \right)^{ij} \left[ \N_i \phi + b \Rc_{ik} \N_k \psi \right]
\left[ \N_j \phi + b \Rc_{jl} \N_l \psi \right]\\
=&\ \left[ \frac{n-1}{n} - 2 a^2 b \right] \int_M \phi^2 + \left[-1 + a^2 b^2
\right] \int_M \Rc_{ij} \N_i \psi \N_j \psi + a^2 \int_M \left( \Rc^{-1}
\right)^{ij} \N_i \phi \N_j \phi\\
=&\ - \frac{n-1}{n} \int_M \phi^2 + \left(\frac{n-1}{n} \right)^2 \int_M \left(
\Rc^{-1} \right)^{ij} \N_i \phi \N_j \phi.
\end{align*}
The inequality follows.  In the case of equality, one observes that in fact
$\psi$ is a solution of
\begin{align*}
 \N^2 \psi - \frac{\gD \phi}{n} g \equiv 0.
\end{align*}
Since $M$ is compact, it follows from \cite{Tashiro} that either $\psi \equiv 0$
or $(M^n,
g)$ is isometric to the round sphere.
\end{proof}
\end{prop}

\begin{prop} \label{Fgeodconv} The functional $F$ is geodesically convex.
\begin{proof} In the positive cone, we use Lemma \ref{metricsplit} and the
geodesic equation to
compute along a geodesic,
\begin{align*}
\frac{d^2}{dt^2} F[u] =&\ \frac{d}{dt} \int_M u_t \left[ K_u - \bar{K}_u
 \right] dA_u\\
=&\  - 4 \pi \frac{d}{dt} \int_M u_t A_u^{-1} dA_u\\
=&\ - 4 \pi \int_M \left[u_{tt} A_u^{-1}  - A_u^{-2} u_t \left(
\int_M 2 u_t dA_u \right) + 2 A_u^{-1} u_t^2 \right] dA_u\\
=&\ 4\pi A_u^{-1} \left[ \int_M \frac{1}{K_u} \brs{\N u_t}^2 dA_u - 2 \left( \int_M u_t^2 dA_u - A_u^{-1} \left( \int_M
u_t dA_u \right)^2 \right) \right]\\
\geq&\ 0,
\end{align*}
where the last line follows from Proposition \ref{andrewsineq}.  In the negative
cone an analogous calculation yields
\begin{align*}
\frac{d^2}{dt^2} F[u_t] =&\ 2 \pi (-\chi) A_u^{-1} \left[ \int_M \frac{1}{- K_u} \brs{\N u_t}^2 dA_u + 2 \left( \int_M u_t^2 dA_u - A_u^{-1} \left( \int_M
u_t dA_u \right)^2 \right) \right]\\
\geq&\ 0.
\end{align*}
\end{proof}
\end{prop}

\begin{prop}  \label{Fflowconvexity} Given $u$ a solution to IGCF, one has
\begin{align*}
\frac{d^2}{dt^2} F[u] \geq&\ 0.
\end{align*}
\begin{proof} If $u \in \gG_1^+$, using (\ref{Sdot}) we have for a solution to IGCF flow
\begin{align*}
\frac{d}{dt} F =&\ - \int_M \left( 1 - \frac{\bar{K}_u}{K_u} \right)^2
K_u dA_u\\
=&\ - \int_M \left[1 - \frac{2 \bar{K}_u}{K_u} + \frac{\bar{K}_u^2}{K_u^2}
\right] K_u dA_u\\
=&\ 2 \pi \chi(M) - \bar{K}_u^2 \int_M \frac{1}{K_u} dA_u.
\end{align*}
Hence we have
\begin{align*}
\frac{d^2}{dt^2} F =&\ \frac{d}{dt} \left[ - \frac{(2 \pi \chi)^2}{A_u^2} \int_M
\frac{1}{K_u} dA_u \right]\\
=&\ \frac{(2 \pi \chi)^2}{A_u^3} \left[ 2 \frac{d}{dt} A_u \int_M \frac{1}{K_u}
dA_u - A_u \frac{d}{dt} \int_M \frac{1}{K_u} dA_u \right]\\
=&\ \frac{(2 \pi \chi)^2}{A_u^3} \left[ 2 \left( - 2 A_u + 2 \bar{K}_u
\int_M \frac{1}{K_u} dA_u \right) \int_M \frac{1}{K_u} dA_u \right.\\
&\ \qquad - A_u \int_M 2 \left(- 1 + \frac{\bar{K}_u}{K_u} \right)
\frac{1}{K_u} dA_u\\
&\ \left. \qquad - A_u \int_M \left( - \frac{2}{K_u} + \frac{2
\bar{K}_u}{K_u^2} + \frac{\bar{K}_u}{K_u^2} \gD \frac{1}{K_u} \right)
dA_u \right]\\
=&\ \frac{(2 \pi \chi)^2}{A_u^3} \left[ 4 \bar{K}_u \left( \int_M
\frac{1}{K_u} dA_u \right)^2 - 4 A_u \bar{K}_u \int_M \frac{1}{K_u^2} dA_u -
 A_u \bar{K}_u \int_M \frac{1}{K_u^2} \gD \frac{1}{K_u} dA_u
\right]\\
=&\ \frac{(2 \pi \chi)^2}{A_u^3} \left[ 4 \bar{K}_u \left( \int_M
\frac{1}{K_u} dA_u \right)^2 - 4 A_u \bar{K}_u \int_M \frac{1}{K_u^2} dA_u + 2
A_u \bar{K}_u \int_M \frac{1}{K_u} \brs{\N \frac{1}{K_u}}^2 dA_u \right]\\
=&\ \frac{16 \pi^3 \chi^3}{A_u^3} \left[ - 2 \int_M \left( \frac{1}{K_u} -
A_u^{-1} \int_M \frac{1}{K_u} \right)^2 dA_u + \int_M \frac{1}{K_u}
\brs{\N \frac{1}{K_u}}^2 dA_u \right]\\
\geq&\ 0,
\end{align*}
where the last line follows from Proposition \ref{andrewsineq}.  A similar
calculation when $u \in \gG_1^-$ yields
\begin{align*}
 \frac{d^2}{dt^2} F =&\ \frac{16 \pi^3 \chi^3}{A_u^3} \left[ 2 \int_M \left(
\frac{1}{K_u} -
A_u^{-1} \int_M \frac{1}{K_u} \right)^2 + \int_M \frac{1}{K_u}
\brs{\N \frac{1}{K_u}}^2 dA_u \right]\\
\geq&\ 0,
\end{align*}
since both terms on the right hand side are nonnegative.
\end{proof}
\end{prop}

\begin{rmk} The inverse Gauss curvature flow also has the remarkable quality of monotonically decreasing distances in $\gG_1^+$.  This is in direct analogy with the fact that the Calabi flow decreases distances in the Mabuchi metric \cite{CalabiChen}.
\end{rmk}

\begin{prop} \label{lengthmonotonicity} Let
$u(s,t)$ $s \in [0,1], t \in [0,T)$
be a smooth two-parameter
family of conformal factors such that for all $s \in [0,1]$, the family
$u(s,\cdot)$ is a solution to IGCF.  Then
\begin{align*}
\frac{d}{dt} L(u(\cdot, t)) \leq 0.
\end{align*}
\begin{proof} For the positive cone, we directly compute
\begin{align*}
\frac{d}{dt} L(u(\cdot, t)) =&\ \frac{d}{dt} \int_0^1 \left[ \int_M \left(
\frac{\del u}{\del s} \right)^2 K_u dA_u \right]^{\frac{1}{2}}ds \\
=&\ \frac{1}{2} \int_0^1 \brs{u_s}_u^{-1} \int_M \left[ 2 u_{ts} u_s K_u -
u_s^2 \gD u_t \right] dA_u ds.
\end{align*}
Now we compute
\begin{align*}
u_{ts} =&\ \frac{\del}{\del s} \left[ \frac{\bar{K}_u}{K_u} - 1\right]\\
=&\ \bar{K}_u \left[ - \frac{1}{A_u K_u} \frac{\del}{\del s} A_u +
\frac{1}{K_u^2} \left( \gD u_s + 2 u_s
K_u \right)\right]\\
=&\ \bar{K}_u \left[ - \frac{1}{A_u K_u} \int_M 2 u_s dA_u + \frac{1}{K_u^2}
\gD u_s + \frac{2 u_s}{K_u}
\right].
\end{align*}
Hence
\begin{align*}
\int_M & 2 u_{ts} u_s K_u dA_u\\
=&\ 2 \bar{K}_u \int_M \left[ - 2 A_u^{-1}
K_u^{-1} \int_M u_s dA_u +  K_u^{-2} \gD u_s+ 2 K_u^{-1} u_s \right] u_s K_u dA_u\\
=&\ 4 \bar{K}_u A_u^{-1} \left\{ A_u \int_M
u_s^2 dA_u - \left[ \int_M u_s dA_u \right]^2 \right\} + 2 \bar{K}_u \int_M
K_u^{-1} u_s \gD u_s dA_u\\
=&\ 4 \bar{K}_u A_u^{-1} \left\{ A_u \int_M
u_s^2 dA_u - \left[ \int_M u_s dA_u \right]^2 \right\}\\
&\ - 2 \bar{K}_u \int_M \left[ K_u^{-1} \brs{\N u_s}^2 + u_s \IP{ \N u_s, \N K_u^{-1}} \right]dA_u.
\end{align*}
Also
\begin{align*}
- \int_M u_s^2 \gD u_t dA_u =&\ 2 \int_M u_s \IP{\N u_t, \N u_s} dA_u\\
=&\ 2 \bar{K}_u \int_M u_s \IP{\N u_s, \N K_u^{-1}} dA_u.
\end{align*}
Collecting these calculations yields
\begin{align*}
\frac{d}{dt} L(u ({\cdot,t})) =&\ - \int_0^1 \bar{K}_u \brs{u_s}_u^{-1} \left[
\int_M K_u^{-1} \brs{\N u_s}^2 dA_u - 2 \int_M u_s^2 + 2 A_u^{-1} \left[ \int_M u_s dA_u \right]^2
\right] ds\\
\leq&\ 0,
\end{align*}
using Proposition \ref{andrewsineq}.  In the negative cone a closely related
calculation yields
\begin{align*}
\frac{d}{dt} L(u(\cdot,t)) =&\ - \int_0^1 (- \bar{K}) \brs{u_s}_u^{-1} \left[
\int_M (-K_u)^{-1} \brs{\N u_s}^2 dA_u + 2 \int_M u_s^2 - 2 A_u^{-1} \left[ \int_M u_s dA_u \right]^2
\right] ds\\
\leq&\ 0,
\end{align*}
using the Cauchy-Schwarz inequality.
\end{proof}
\end{prop}

\subsection{Higher genus surfaces}

In this section we analyze the $IGCF$ in the space of conformal metrics of negative curvature.  To simplify notation we drop the subscript $u$
and write (\ref{neggradflow}) as
\begin{align} \label{surfnce}
 \dt u =&\ \frac{K - \bar{K}}{K}.
\end{align}
In the following, the curvature and other metric-dependent quantities will be understood to be with respect to $g$; all quantities with a subscript $0$ are with
with respect to a fixed background metric $g_0$ which we may take to be the initial metric.  \\

\begin{lemma} \label{surfnceKflow} Let $u$ be a solution to (\ref{surfnce}).
Then
\begin{align*}
 \dt K =&\ - \frac{\bar{K}}{K^2} \gD K + 2 \frac{\bar{K}}{K^3} \brs{\N K}^2 + 2
(\bar{K} - K).
\end{align*}

 \begin{proof}
  We directly compute
 \begin{align*}
  \dt K =&\ \dt e^{-2u} \left[ - \gD_{g_0} u + K_0 \right]\\
  =&\ - 2 \dot{u} K - e^{-2u} \gD_{g_0} \dot{u}\\
  =&\ - 2 K \left[ \frac{K - \bar{K}}{K} \right] - \gD_{g_u} \left[\frac{K -
\bar{K}}{K} \right]\\
  =&\ \bar{K} \gD \frac{1}{K} + 2 (\bar{K} - K)\\
  =&\ - \frac{\bar{K}}{K^2} \gD K + 2 \frac{\bar{K}}{K^3} \brs{\N K}^2 + 2
(\bar{K} - K).
 \end{align*}
 \end{proof}
\end{lemma}

\begin{prop} \label{ncecurvest} Let $u$ be a solution to (\ref{surfnce}).
Then for $0 \leq t < T$,
\begin{align*}
\inf_M K_0 \leq K \leq \sup_M K_0.
\end{align*}
\begin{proof} We apply the maximum principle to the result of Lemma
\ref{surfnceKflow}.  At a minimum point, certainly $-\frac{\bar{K}}{K^2} \gD K
\geq 0$ (since $\bar{K} < 0$), $\N K = 0$, and $K < \bar{K}$.  It follows that
the minimum of $K$ is nondecreasing along the flow.  A similar argument shows
that the maximum is nonincreasing.
\end{proof}
\end{prop}

\begin{prop} \label{nceuest} Let $u$ be a solution to (\ref{surfnce}).  There
exists a constant $C = C(M, u_0)$ such that for all $0 \leq t < T$,
\begin{align*}
 \brs{u} \leq C.
\end{align*}
\begin{proof} We apply the maximum principle directly to (\ref{surfnce}).  At a
spacetime minimum point for $u$, we have $\gD_{g_0} u \geq 0$, and hence
 \begin{align*}
  \dt u =&\ 1 - \frac{\bar{K}}{K}\\
  =&\ 1 - \frac{\bar{K}}{e^{-2u}(-\Delta_0 u + K_0)} \\
  \geq&\ 1 - \frac{\bar{K}}{e^{-2 u} K_0}\\
  =&\ 1 - \frac{\bar{K}}{K_0} e^{2 u}\\
  \geq&\ 0,
 \end{align*}
if $u(x,t) \leq \frac{1}{2} \log \frac{\sup K_0}{\bar{K}}$.  Similarly, at a
spacetime maximum one has $\gD_{g_0} u \leq 0$ and hence
\begin{align*}
 \dt u =&\ 1 - \frac{\bar{K}}{K}\\
  \leq&\ 1 - \frac{\bar{K}}{e^{-2 u} K_0}\\
  =&\ 1 - \frac{\bar{K}}{K_0} e^{2 u}\\
  \leq&\ 0,
\end{align*}
if $u(x,t) \geq \frac{1}{2} \log \frac{\inf K_0}{\bar{K}}$.  The result follows.
\end{proof}
\end{prop}

Next, we use the {\em a priori} estimates of Propositions \ref{ncecurvest} and
\ref{nceuest} to prove higher order estimates. In what follows we will assume there
is a constant $\gL$ such that
\begin{align} \label{surfncelambda}
 -\gL \leq K \leq -\gL^{-1}, \qquad \brs{u} \leq \gL.
\end{align}
Furthermore, given a smooth solution to (\ref{surfnce}) satisfying these bounds, we adopt the
notation that $X_i$ refers to any quantity which is uniformly controlled along
the flow in terms of $\gL$.

We begin with a gradient estimate for solutions:

\begin{lemma} \label{surfncegrad} Let $u$ be a solution to
(\ref{surfnce}) satisfying (\ref{surfncelambda}).  Then
 \begin{align*}
  \dt \brs{\N_0 u}^2 \leq&\ - \frac{\bar{K}}{K^2} \gD_0 \brs{\N u}^2_0 + \frac{2
\bar{K}}{K^2}e^{-2u} \brs{\N^2_0 u}_0^2 + C_1 (\brs{\N u}_0^2 + 1),
 \end{align*}
 where $C_1$ depends on the initial data and $\Lambda$.

\begin{proof} We first compute that
\begin{align*}
 \dt \N_0 u =&\ \N_0 \dt u\\
 =&\ - \bar{K} \N_0 K^{-1}\\
 =&\ \frac{\bar{K}}{K^2} \N_0 \left[ e^{-2 u} \left( - \gD_0 u + K_0 \right)
\right]\\
 =&\ \frac{\bar{K}}{K^2} e^{-2u} \left[ - \N_0 \gD_0 u \right] - 2
\frac{\bar{K}}{K}  \N_0 u + \frac{\bar{K}}{K^2} e^{-2 u} \N_0 K_0\\
=&\ -\frac{\bar{K}}{K^2} \gD \N_0 u + X_1 * \N_i u + X_2.
\end{align*}
\end{proof}

\end{lemma}

\begin{lemma} \label{surfncehess} Let $u$ be a solution to
(\ref{surfnce}) satisfying (\ref{surfncelambda}).
Given $Z$ a smooth vector field on $M$, one has
\begin{align*}
 \dt \N^0_Z \N^0_Z u \leq&\ - \frac{\bar{K}}{K^2} \gD_0 \N^0_Z \N^0_Z u + \N_0 Z
* \N_0^3 u + X_1 *
\N^0
\N^0 u + X_2 * (\N u)^{*2} + X_3.
\end{align*}
\begin{proof}
First we compute
\begin{align*}
\dt \N^0_Z \N^0_Z u =&\ \N^0_Z \N^0_Z \left[ \frac{K - \bar{K}}{K} \right]\\
=&\ - \bar{K} \N^0_Z \N^0_Z K^{-1}\\
=&\ \frac{\bar{K}}{K^2} \N^0_Z \N^0_Z K - \frac{2 \bar{K}}{K^3} \N_Z K \otimes
\N_Z K\\
=&\ \frac{\bar{K}}{K^2} \N^0_Z \N^0_Z \left[ e^{-2u} \left( - \gD_0 u + K_0
\right)
\right] - \frac{2 \bar{K}}{K^3} \N_Z K \otimes \N_Z K\\
=&\ \frac{\bar{K}}{K^2} \left[ \left( \N^0_Z \N^0_Z e^{-2u} \right) e^{2u} K -2
\N^0_Z e^{-2u} \N_Z \left(e^{2u} K \right) + e^{-2u} \N^0_Z \N^0_Z \left( -
\gD_0 u + K_0 \right) \right]\\
&\ - \frac{2 \bar{K}}{K^3} \N_Z K \otimes \N_Z K.\\
\leq&\ - \frac{\bar{K}}{K^2} e^{-2 u} \N^0_Z \N^0_Z \gD_0 u + X_1 * \N^0_Z
\N^0_Z u + X_2 * (\N u)^{*2} + X_3,
\end{align*}
where we applied the Cauchy-Schwarz inequality in the final line.  Next, we
commute derivatives to yield
\begin{align*}
 \N^0_Z \N^0_Z \gD_0 u =&\ \N^0_Z \N^0_Z \N^0_{e_i} \N^0_{e_i} u\\
 =&\ \N^4 u \left( Z, Z, e_i, e_i \right)\\
 =&\ \N^4 u (e_i, e_i, Z, Z) + K_0 * \N^2_0 u + \N K_0 * \N u\\
 =&\ \gD_0 \left( \N_0^2 u (Z,Z) \right) + \N_0 Z * \N_0^3 u + \N^2_0 u * \left(
\N_0^2 Z + (\N_0 Z)^{*2} \right)\\
 &\ + K_0 * \N^2_0 u + \N K_0 * \N u.
\end{align*}

\end{proof}
\end{lemma}

\begin{prop} \label{ncehessest} Given a controlled solution to (\ref{surfnce})
on $[0,T)$ there
exists a constant $C = C(\gL,u_0)$ such that
\begin{align*}
 \sup_{M \times [0,T)} \brs{\N^2_0 u} \leq C.
\end{align*}
\begin{proof} Fix some constant $A$ and consider the function
\begin{align*}
 \gb(p) = \max_{X \in T_pM \backslash \{0\}} \frac{\N_0^2 u (X,X)}{\brs{X}^2}
\end{align*}
This is continuous, and an upper bound for $\gb$ implies an upper bound for the
Hessian of $u$.  Fix a constant $A$ and consider the function
\begin{align*}
 \Phi(x,t) = t \gb + A \brs{\N u}_0^2.
\end{align*}
We claim that for $A$ chosen sufficiently large that there is an a priori bound
for an interior maximum on $[0,1]$.  Suppose some point $(x,t)$ is such an
interior spacetime maximum for $\Phi$.  Fix a unit vector $Z \in T_x M$
realizing the supremum in the definition of $\gb(x)$.  We may extend $Z$ in a
neighborhood of $x$ by parallel transport along radial geodesics.  This yields
\begin{align*}
 \brs{Z} =&\ 1 \mbox{ on } B_{\ge}(x)\\
 \N^0_X Z (x) =&\ 0 \mbox{ for all } X\\
 \brs{\N_0^2 Z} (x) \leq&\ C(g_0).
\end{align*}
We may extend $Z$ to all of $M$ by multiplying by a cutoff function.  We observe
that the function
\begin{align*}
 \Psi(x,t) =&\ t \N^2_0(Z,Z) + A \brs{\N u}_0^2
\end{align*}
also has a spacetime maximum at $(x,t)$.
Combining Lemmas \ref{surfncegrad} and \ref{surfncehess} yields that, at $x$,
\begin{align*}
0 \leq&\ \dt \Psi + \frac{\bar{K}}{K^2} \gD \Psi \\
\leq&\ \N_0 Z
* \N_0^3 u + X_1 *
\N^0
\N^0 u + X_2 * (\N u)^{*2} + X_3\\
 &\ + A \left[ \frac{2 \bar{K}}{K^2} e^{-2 u} \brs{\N_0^2 u}^2_0 + C \brs{\N
u}^2_0 + C \right]\\
\leq&\ - \gd \brs{\N_0^2 u}^2_0 + C,
\end{align*}
where the last line follows by choosing $A$ sufficiently large with respect to
$\gL$ and using the a priori gradient bound.  This implies an a priori upper
bound for the Hessian of $u$ at $(x,t)$, and hence since $t \leq 1$ an a priori
bound for $\Psi$ at $(x,t)$.  This yields an a priori upper
bound for $\N^0_Z \N^0_Z u$ for any interior time $t > \ge$.  Combined with some
ineffective estimate depending on the given solution for $[0,\ge)$ yields an a
priori upper bound on $[0,T)$.  Since the Laplacian of $u$ is uniformly bounded
below, this then yields the full a priori Hessian bound.
\end{proof}
\end{prop}

\begin{prop} \label{nceltethm} Let $(M^2, g_0)$ be a compact Riemann surface such
that $K_0 < 0$.
Given $u \in \gG_1^-$, the solution to (\ref{surfnce}) exists for all time
and converges exponentially to a metric of constant negative scalar curvature.
\begin{proof} Propositions \ref{nceuest} and \ref{ncecurvest} guarantee uniform
estimates on $u$, $K$ and $\frac{1}{K}$.  Proposition \ref{ncehessest} then
implies a uniform estimate for the Hessian of $u$.  We now observe that the
operator $\Phi(u) =
\frac{\bar{K} - K}{K}$ is convex, and with the uniform upper and lower bounds on
curvature, is uniformly elliptic.  Thus the Evans-Krylov theorem \cite{Evans}, \cite{Krylov} yields
an a priori $C^{2,\ga}$ estimate for $u$.  Schauder estimates can then be
applied to obtain estimates of every $C^{k,\ga}$ norm of $u$.  It follows that
the solution exists on
$[0,\infty)$ and every sequence of times approaching infinity admits a
subsequence such that $\{u_{t_i}\}$ converges to a limiting function
$u_{\infty}$.  Using that the flow is the gradient flow for $F$ it follows easily that the limiting metric
$u_{\infty}$ has constant curvature, and since this metric is unique the whole flow converges to $u_{\infty}$.
\end{proof}
\end{prop}

\subsection{The sphere}

We now consider the case of $K_0 > 0$, and so $M \cong S^2$.  In this case we
are studying the flow
\begin{align} \label{surfacepce}
 \dt u =&\ \frac{\bar{K} - K}{K}.
\end{align}

\subsubsection{Evolution Equations}

To begin we build up some evolution equations.  First we rewrite the evolution
of $u$ in terms of the linearized operator.

\begin{lemma} \label{highreglem1} Given $u$ a solution to (\ref{surfacepce})
we
have
 \begin{align*}
  \dt u =&\ \frac{\bar{K}}{K^2} \gD u - 1 + \frac{2 \bar{K}}{K} - \frac{e^{-2 u}
K_0 \bar{K}}{K^2}.
 \end{align*}
\begin{proof} To begin we compute
\begin{align*}
 - \frac{\bar{K}}{K^2} \gD u =&\ \frac{\bar{K}}{K^2} \left[ K - e^{-2 u} K_0
\right] = \frac{\bar{K}}{K} - \frac{e^{-2 u} K_0 \bar{K}}{K^2}.
\end{align*}
Hence using (\ref{surfacepce}) we have
\begin{align*}
 \dt u - \frac{\bar{K}}{K^2} \gD u =&\ \frac{\bar{K} - K}{K} + \frac{\bar{K}}{K}
-
\frac{e^{-2 u} K_0 \bar{K}}{K^2}\\
 =&\ -1 + \frac{2 \bar{K}}{K} - \frac{e^{-2 u} K_0 \bar{K}}{K^2},
\end{align*}
as required.
\end{proof}
\end{lemma}

\begin{lemma} Let $u$ be a solution to (\ref{surfacepce}).  Then
\begin{align} \label{pcecurvev}
 \dt K =&\ \frac{\bar{K}}{K^2} \gD K - 2 \frac{\bar{K} \brs{\N K}^2}{K^3} + 2
\left(K - \bar{K} \right).
\end{align}
\begin{proof} We directly compute
 \begin{align*}
  \dt K =&\ \dt e^{-2 u} \left[ - \gD_{g_0} u + K_0 \right]\\
  =&\ - 2 \dot{u} K - e^{-2 u} \gD_{g_0} \dot{u}\\
  =&\ - 2 K \left[ \frac{\bar{K} - K}{K} \right] - \bar{K} \gD \frac{1}{K}\\
  =&\ \frac{\bar{K}}{K^2} \gD K - 2 \frac{\bar{K} \brs{\N K}^2}{K^3} + 2 \left(K
- \bar{K} \right),
 \end{align*}
as required.
\end{proof}
\end{lemma}

\begin{lemma} \label{highreglem2} Given $u$ a solution to (\ref{surfacepce})
we
have
 \begin{align*}
  \dt K^{-1} =&\ \frac{\bar{K}}{K^2} \gD K^{-1}- 2 K^{-1} + 2
\frac{\bar{K}}{K^2}.
 \end{align*}
 \begin{proof} We compute
 \begin{align*}
  \dt K^{-1} =&\ - \frac{\dt K}{K^2}\\
  =&\ \frac{1}{K^2} \left[ 2 u_t K + e^{-2 u} \gD_{g_0} u_t \right]\\
  =&\ \frac{\bar{K}}{K^2} \gD K^{-1} - 2 K^{-1} + 2 \frac{\bar{K}}{K^2},
 \end{align*}
 as required.
 \end{proof}
\end{lemma}

\begin{lemma} \label{n2areaestimate} Let $u$ be a solution to
(\ref{surfacepce}).  Then
\begin{align*}
A(t) = A(0) e^{2(F[u_0] - F[u_t])}.
\end{align*}
\begin{proof} We observe
\begin{align*}
\frac{d}{dt} A =&\ \dt \int_M e^{2u} dV_0\\
=&\ 2 \int_M \frac{\bar{K} - K}{K} dV_u\\
=&\ A \left[ - 2 + \frac{\bar{K}}{A^2} \int_M \frac{1}{K} dV_u \right]\\
=&\ - 2 A \frac{dF}{dt}.
\end{align*}
Integrating this ODE we conclude the result.
\end{proof}
\end{lemma}

\begin{lemma} \label{n2dirichlet} Let $u$ be a solution to (\ref{surfacepce}).
 Then
\begin{align*}
\int_M \left( \brs{\N u}^2 + 2 K_0 u \right) dV_0 =&\ \int_M \left( \brs{\N u_0}^2
+ 2 K_0 u_0 \right) dV_0.
\end{align*}
\begin{proof} One observes that
\begin{align*}
\int_M \left( \brs{\N u}^2 + 2 K_0 u \right) dV_0 = F[u] + \log A,
\end{align*}
and the result follows directly from the calculation of Lemma
\ref{n2areaestimate}.
\end{proof}
\end{lemma}

\subsubsection{A priori estimates}

\begin{prop} \label{pcecub} Let $u$ be a solution to (\ref{surfacepce}).  For
all smooth existence times $t$ of the flow one has
\begin{align*}
\sup_{M \times \{t\}} K \leq \sup_M K_0 e^{2 t}.
\end{align*}
\begin{proof} This follows directly from the maximum principle applied to
(\ref{pcecurvev}).
\end{proof}
\end{prop}

\begin{prop} \label{pceulb} Let $u$ be a solution to (\ref{surfacepce}).  For
all smooth existence times $t$ of the flow one has
 \begin{align*}
  \inf_{M \times\{t\}} u \geq \inf_M u_0 - t.
 \end{align*}
\begin{proof} We observe that at a spacetime minimum for $u$, one has $\gD_{g_0}
u \geq 0$, and hence
\begin{align*}
 \dt u =&\ - 1 + \frac{\bar{K}}{K}\\
 \geq&\ - 1 + \frac{\bar{K} e^{2 u}}{K_0}\\
 \geq&\ - 1.
\end{align*}
The result follows from the maximum principle.
\end{proof}
\end{prop}

\begin{prop} \label{pcegradL2} Let $u$ be a solution to (\ref{surfacepce}).
There exists a constant $C = C(u_0)$ such that for
all smooth existence times $t$ of the flow one has
 \begin{align*}
\brs{\brs{\N u(t)}}_{L^2}^2 \leq C \left[1 + t \right]
 \end{align*}
\begin{proof} We use Lemma \ref{n2dirichlet} with the estimate of Proposition
\ref{pceulb} to yield
\begin{align*}
\brs{\brs{\N u}}_{L^2}^2 =&\ \int_M \left(\brs{\N u}^2 + 2 K_0 u - 2 K_0 u \right)
dV_0\\
=&\ \int_M \left(\brs{\N u_0}^2 +2 K_0 u_0 \right) dV_0 - 2 K_0 \int_M u dV_0\\
\leq&\ C - C \inf u\\
\leq&\ C \left(1 + t \right).
\end{align*}
\end{proof}
\end{prop}

\begin{prop} \label{pceuub} Let $u$ be a solution to (\ref{surfacepce}) on $[0,T)$.  There exists a constant $C =
C(u_0,T)$ such one has
\begin{align*}
\sup_{M \times [0,T)} \brs{u} \leq C.
\end{align*}
\begin{proof} First from Proposition \ref{pcegradL2}  there is a time-dependent
bound on $\brs{\brs{\N u}}_{L^2}$.  By the Moser-Trudinger inequality we obtain
a uniform estimate of $\int_M e^{4 \brs{u}} dV_0$.  We now claim that there is a
uniform constant $R > 0$ so that
\begin{align*}
\sup_{x \in S^2, t \in [0,T)} \int_{B_R(x_0)} \brs{K_t} dV_t < 2 \pi.
\end{align*}
Using Proposition \ref{pcecub} we estimate
\begin{align*}
 \int_{B_R(x_0)} \brs{K_t} dV_t =&\ \int_{B_R(x_0)} K_t e^{2 u_t} dV_0\\
\leq&\ C \int_{B_R(x_0)} e^{2 u_t} dV_0\\
\leq&\ C \left[ \int_{B_r(x_0)} e^{4 \brs{u_t}} dV_0 \right]^{\frac{1}{2}}
\left[ \int_{B_R(x_0)} dV_0 \right]^{\frac{1}{2}}\\
\leq&\ C R\\
<&\ 2 \pi,
\end{align*}
where the last inequality follows by choosing $R$ small with respect to $C$.
Invoking \cite{Struwe} Theorem 3.2 we conclude a uniform $H_2^2$ bound for $u$,
which by the Sobolev inequality implies the uniform bound for $u$.
\end{proof}
\end{prop}

\begin{prop} \label{pceclb} Let $u$ be a solution to (\ref{surfacepce}) on
$[0,T)$ such that
$\brs{u} \leq \gL$.  There exists a constant $C = C(T,u_0)$ such that
 \begin{align*}
K^{-1} \leq C.
 \end{align*}
\begin{proof} Let $\Phi = t K^{-1} + A u$, where $A > 0$ is a constant yet to be
determined.  Using Lemmas \ref{highreglem1}, \ref{highreglem2} and Proposition
\ref{pceuub} we have
\begin{align*}
\left( \dt - \frac{\bar{K}}{K^2} \gD \right) \Phi =&\ - 2 K^{-1} + 2 \bar{K}
K^{-2} + A \left[ -1 + 2 \bar{K} K^{-1} - e^{-2 u} K_0 \bar{K} K^{-2} \right]\\
\leq&\ K^{-2} \left[ C + C A K - \gd A \right].
\end{align*}
where the constant $\gd > 0$ is determined by the upper bound for $u$.  If we
choose $A$ sufficiently large with respect to $\gd$, then at a sufficiently
large maximum for $K^{-1}$ we obtain
\begin{align*}
C + C A K - \gd A \leq C - \frac{\gd}{2} A \leq 0.
\end{align*}
The result follows from the maximum principle.
\end{proof}
\end{prop}

\begin{prop}\label{PCflow} Let $(M^2, g_0)$ be a compact Riemann surface such that $K_0 > 0$.
Given $u \in C^{\infty}(M)$, the solution to (\ref{surfacepce}) exists for all
time.
\begin{proof} Combining Propositions \ref{pceuub}, \ref{pcecub}, and
\ref{pceclb}, we obtain uniform estimates on $u$, $K$, and $K^{-1}$ for any
finite existence time $T$.  Given these the higher order estimates follow as in
the proof of Proposition \ref{ncehessest} and Theorem \ref{nceltethm}, and
so the long time existence follows.
\end{proof}
\end{prop}

\begin{prop} \label{PCconv} Given $(S^2, g_{S^2})$, and $u \in \gG_1^+$, the solution to (\ref{surfacepce}) exists for all time and converges weakly in the distance topology to a minimizer for $F$ in the completion $(\bar{\gG}_1^+, \bar{d})$.
\begin{proof} We provide a sketch of the proof, as the key ingredients have been established already and the result follows formally from prior results.  The principal tool we require is a general result \cite{Bacak} concerning convergence of weak gradient flows in metric spaces.

As we have established that $(\gG_1^+,d)$ is NPC and convex by unique regularizable $C^{1,1}$ geodesics, the argument of (\cite{SKEMM} Lemma 5.9) shows that the completion $(\bar{\gG}_1^+, \bar{d})$ is also an NPC space.  Moreover, let $\bar{F}$ denote the extension of $F$ to $\bar{\gG}_1^+$ via its canonical lower-semicontinuous extension.  Following the argument of (\cite{SKEMM} Lemma 5.15), it follows that $\bar{F}$ is geodesically convex.  Moreover, we know that the minimum of $F$ is attained by any constant curvature metric, which exists by assumption.  It follows that these also realize the minimum of $\bar{F}$.  Hence we have verified the setup of (\cite{Mayer} Theorem 1.13), guaranteeing the existence of a global weak solution to the gradient flow of $\bar{F}$ with arbitrary initial data.  Moreover, following the argument of (\cite{SCCKEMM} Theorem 1.1) we can verify that the smooth global solutions of Proposition \ref{PCflow} coincide with the weak solutions constructed via (\cite{Mayer} Theorem 1.13).  The convergence of the weak flows, and hence the smooth flows, to a minimizer for $\bar{F}$ in the weak distance topology now follows from \cite{Bacak} Theorem 1.5.
\end{proof}
\end{prop}

\begin{proof} [Proof of Theorem \ref{flowthm}] The required results follow from Propositions \ref{Fflowconvexity}, \ref{lengthmonotonicity}, \ref{nceltethm}, \ref{PCflow} and \ref{PCconv}.
\end{proof}

%

\end{document}